\documentclass[12pt]{amsart}
\usepackage{hyperref} 
\usepackage{amsmath, amsthm, amssymb}
\usepackage{hyperref} 
\usepackage{enumerate}
\usepackage{verbatim}
\usepackage{esint}
\usepackage[T1]{fontenc}

\usepackage{tikz,amsthm,amsmath,amstext,amssymb,amscd,epsfig,euscript, mathrsfs,
 dsfont,pspicture,
multicol,graphpap,graphics,graphicx,times,enumerate,subfig,
sidecap,
wrapfig,color,pict2e}
\usepackage{setspace}

\numberwithin{equation}{section}

\title{Quantitative decompositions of Lipschitz mappings into metric spaces}
\author{Guy C. David}
\address{Department of Mathematical Sciences\\ Ball State University, Muncie, IN 47306}
\email{gcdavid@bsu.edu}

\author{Raanan Schul}
\address{Department of Mathematics\\ Stony Brook University\\ Stony Brook, NY 11794-3651}
\email{schul@math.sunysb.edu}

\date{\today}
\thanks{G.~ C.~ David was partially supported by the National Science Foundation under Grant No. DMS-1758709. R.~ Schul was partially supported by the National Science Foundation under Grants No. DMS-1763973.}


\begin{document}

\theoremstyle{plain}
\newtheorem{theorem}{Theorem}
\newtheorem{exercise}{Exercise}
\newtheorem{corollary}[theorem]{Corollary}
\newtheorem{scholium}[theorem]{Scholium}
\newtheorem{claim}[theorem]{Claim}
\newtheorem{lemma}[theorem]{Lemma}
\newtheorem{sublemma}[theorem]{Lemma}
\newtheorem{proposition}[theorem]{Proposition}
\newtheorem{conjecture}[theorem]{Conjecture}
\newtheorem{maintheorem}{Theorem}
\newtheorem{maincor}[maintheorem]{Corollary}
\newtheorem{mainproposition}[maintheorem]{Proposition}
\renewcommand{\themaintheorem}{\Alph{maintheorem}}

\theoremstyle{definition}
\newtheorem{fact}[theorem]{Fact}
\newtheorem{example}[theorem]{Example}
\newtheorem{definition}[theorem]{Definition}
\newtheorem{remark}[theorem]{Remark}
\newtheorem{question}[theorem]{Question}

\numberwithin{equation}{section}
\numberwithin{theorem}{section}

\newcommand{\cG}{\mathcal{G}}
\newcommand{\RR}{\mathbb{R}}
\newcommand{\HH}{\mathcal{H}}
\newcommand{\LIP}{\textnormal{LIP}}
\newcommand{\Lip}{\textnormal{Lip}}
\newcommand{\Tan}{\textnormal{Tan}}
\newcommand{\length}{\textnormal{length}}
\newcommand{\dist}{\textnormal{dist}}
\newcommand{\diam}{\textnormal{diam}}
\newcommand{\vol}{\textnormal{vol}}
\newcommand{\rad}{\textnormal{rad}}
\newcommand{\side}{\textnormal{side}}

\def\bA{{\mathbb{A}}}
\def\bB{{\mathbb{B}}}
\def\bC{{\mathbb{C}}}
\def\bD{{\mathbb{D}}}
\def\bR{{\mathbb{R}}}
\def\bS{{\mathbb{S}}}
\def\bO{{\mathbb{O}}}
\def\bE{{\mathbb{E}}}
\def\bF{{\mathbb{F}}}
\def\bH{{\mathbb{H}}}
\def\bI{{\mathbb{I}}}
\def\bT{{\mathbb{T}}}
\def\bZ{{\mathbb{Z}}}
\def\bX{{\mathbb{X}}}
\def\bP{{\mathbb{P}}}
\def\bN{{\mathbb{N}}}
\def\bQ{{\mathbb{Q}}}
\def\bK{{\mathbb{K}}}
\def\bG{{\mathbb{G}}}

\def\nrj{{\mathcal{E}}}
\def\cA{{\mathscr{A}}}
\def\cB{{\mathscr{B}}}
\def\cC{{\mathscr{C}}}
\def\cD{{\mathscr{D}}}
\def\cE{{\mathscr{E}}}
\def\cF{{\mathscr{F}}}
\def\cB{{\mathscr{G}}}
\def\cH{{\mathscr{H}}}
\def\cI{{\mathscr{I}}}
\def\cJ{{\mathscr{J}}}
\def\cK{{\mathscr{K}}}
\def\Layer{{\rm Layer}}
\def\cM{{\mathscr{M}}}
\def\cN{{\mathscr{N}}}
\def\cO{{\mathscr{O}}}
\def\cP{{\mathscr{P}}}
\def\cQ{{\mathscr{Q}}}
\def\cR{{\mathscr{R}}}
\def\cS{{\mathscr{S}}}
\def\Up{{\rm Up}}
\def\cU{{\mathscr{U}}}
\def\cV{{\mathscr{V}}}
\def\cW{{\mathscr{W}}}
\def\cX{{\mathscr{X}}}
\def\cY{{\mathscr{Y}}}
\def\cZ{{\mathscr{Z}}}

  \def\del{\partial}
  \def\diam{{\rm diam}}
	\def\VV{{\mathcal{V}}}
	\def\FF{{\mathcal{F}}}
	\def\QQ{{\mathcal{Q}}}
	\def\BB{{\mathcal{B}}}
	\def\XX{{\mathcal{X}}}
	\def\PP{{\mathcal{P}}}

  \def\del{\partial}
  \def\diam{{\rm diam}}
	\def\image{{\rm Image}}
	\def\domain{{\rm Domain}}
  \def\dist{{\rm dist}}
	\newcommand{\Gr}{\mathbf{Gr}}
\newcommand{\md}{\textnormal{md}}
\newcommand{\vspan}{\textnormal{span}}

\newcommand{\RS}[1]{{  \color{red} \textbf{Raanan:} #1}}

\begin{abstract}
We study the quantitative properties of Lipschitz mappings from Euclidean spaces into metric spaces. We prove that it is always possible to decompose the domain of such a mapping into pieces on which the mapping ``behaves like a projection mapping'' along with a ``garbage set'' that is \textbf{arbitrarily small} in an appropriate sense. Moreover, our control is quantitative, i.e., independent of both the particular mapping and the metric space it maps into.
This improves a theorem of Azzam-Schul from the paper ``Hard Sard'', and answers a question left open in that paper. The proof uses ideas of quantitative differentiation, as well as a detailed study of how to supplement Lipschitz mappings by additional coordinates to form bi-Lipschitz mappings.
\end{abstract}
\maketitle
\tableofcontents

\section{Introduction}

In this paper, we study the quantitative properties of Lipschitz mappings from Euclidean spaces into metric spaces. We prove that it is always possible to decompose the domain into pieces on which the mapping ``behaves like a projection mapping'', along with a ``garbage set'' that is arbitrarily small in an appropriate sense. Moreover, this decomposition is quantitative. This improves the main result of \cite{AS12} and answers the question posed in Remark 6.15 of that paper.

Before stating our new results precisely, we begin with some background.

\subsection{Background}
This paper is concerned with the quantitative structure of Lipschitz mappings from Euclidean spaces into metric spaces. In particular, we are interested in decomposing the domain of a given Lipschitz mapping $f$, which in our case will be the unit cube $Q_0=[0,1]^d$, into a finite number of pieces on which $f$ is well-behaved in some specific way, along with a ``garbage'' set which is small in some sense.

Moreover, we will aim for our decompositions to satisfy these properties in a \textit{quantitative} way. That is, we will aim to control
\begin{itemize}
\item the number of pieces in the decomposition,
\item the properties and bounds that $f$ will satisfy on each piece, and
\item the size of the garbage set
\end{itemize}
in a way which is \textit{independent} of the particular mapping $f$ or the metric space it maps into.

As a starting point, we recall Rademacher's theorem, which states that Lipschitz mappings between Euclidean spaces are differentiable almost everywhere. Using this, one can show that if $f:Q_0\rightarrow \RR^k$ is Lipschitz, then there are countably many sets $E_i$ on which $f$ is bi-Lipschitz and such that
$$ \HH^d(f(Q_0 \setminus \cup_i E_i))=0.$$
(See \cite[Lemma 3.2.2]{Fe}.) Here, $\HH^d$ denotes $d$-dimensional Hausdorff measure. Note that this result is interesting only if $k\geq d$.

This gives a decomposition of the domain of $f$ into nice pieces $E_i$, pieces on which $f$ acts as a bi-Lipschitz homeomorphism, and a garbage set $Q_0 \setminus \cup_i E_i$ whose image has zero measure. On the other hand, this decomposition is not quantitative: there is no control on the number of pieces, which may be infinite, or on the bi-Lipschitz constants for $f$ on each piece. It is much more difficult to obtain such control, but this was accomplished by work of David, Jones, and Semmes beginning in the late 1980's, motivated by applications to singular integrals and uniform rectifiability \cite{Da88, Jo88, Se00}. A very general statement in this vein, allowing arbitrary metric space targets, was proven by Schul \cite{Sc09}:

\begin{theorem}[Schul, Theorem 1.1 of \cite{Sc09}]\label{thm:lipbilip}
Given $\alpha \in (0,1)$ and $d\in\mathbb{N}$, there are constants $M=M(\alpha,d)$ and $L=L(d)$ with the following properties:

Let $X$ be any metric space, $Q_0$ the unit cube of $\RR^d$, and $f\colon Q_0\rightarrow X$ a $1$-Lipschitz map. Then there are sets $E_1, \dots, E_M \subseteq Q_0$ such that
\begin{enumerate}[(i)]
\item $f|_{E_i}$ is $\alpha^{-1}$-bi-Lipschitz for each $i\in\{1,\dots, M\}$, and
\item $\HH^d_\infty(f(Q_0 \setminus \cup_i E_i)) < L\alpha.$
\end{enumerate}
\end{theorem}
(In fact, \cite[Theorem 1.1]{Sc09} is slightly stronger than what we have stated here.) Note that here the smallness of the garbage set $Q_0 \setminus \cup E_i$ is measured by the $d$-dimensional Hausdorff content $\HH^d_\infty$ of its image, rather than $d$-dimensional Hausdorff measure. See subsection \ref{subsec:hausdorff} for these definitions.

In a different direction, which we will not pursue here, one may consider the existence of such decompositions of Lipschitz mappings when the domain is a more general, non-Euclidean, metric space. We point the reader to \cite{Meyerson, GCD16, Li, LLR, GCDKin} for more on this interesting area.

Theorem \ref{thm:lipbilip} gives a complete answer to the question of finding good decompositions of Lipschitz mappings from $[0,1]^d$, in the case when the image has positive $d$-dimensional Hausdorff measure.

The question of finding quantitative decompositions of Lipschitz mappings becomes more difficult when the image dimension is smaller than the domain dimension, e.g., for mappings $f\colon \RR^3 \rightarrow \RR^2$. Here, one cannot of course expect any bi-Lipschitz behavior in the mapping on a set of positive $3$-dimensional measure. (While Theorem \ref{thm:lipbilip} applies, it is vacuous.) 

A natural question in this dimension-lowering setting is whether $f$ can be decomposed, quantitatively, into pieces on which it ``looks like'' a projection mapping. In \cite{AS12}, Azzam and Schul give a necessary and sufficient condition for mapping to admit a large piece on which it ``looks like'' a projection from $\RR^{n+m}$ to $\RR^n$. (But see Remark \ref{rmk:defdifferences} below.)

Their condition involves the following notion. 
\begin{definition}\label{def:mappingcontent}
Let $E\subseteq Q_0 = [0,1]^{n+m}$ be a set. We define the ``$(n,m)$-mapping content'' of $f$ on $E$  (or just the ``mapping content'' if $n,m$ are understood) as
$$\HH^{n,m}_\infty(f,E):= \inf \sum_{Q_i} \HH^n_\infty(f(Q_i))\side(Q_i)^m,$$
where the infimum is taken over all coverings $\{Q_i\}$ of $E$ by dyadic cubes with disjoint interiors in $Q_0$.
\end{definition}
As discussed in \cite{AS12}, $\HH^{n,m}_\infty$ serves in some sense as a ``coarse'' substitute for the $L^1$-norm of the Jacobian of $f$. 

What is meant by ``looking like a projection'' is encapsulated by the notion of a ``Hard Sard pair'' for a mapping $f$ defined on $Q_0$. This pair consists of a set $E$ and a globally bi-Lipschitz change of coordinates, such that, in these coordinates, $f|_E$ becomes a mapping that is constant on ``vertical'' $m$-planes and bi-Lipschitz on ``horizontal'' $n$-planes. (The name is adapted from the title of \cite{AS12}, and is meant to evoke quantitative analogs of Sard's theorem.)

\begin{definition}\label{def:HSpair}

Let $E\subseteq Q_0=[0,1]^{n+m}$ and $g:E\rightarrow \RR^{n+m}$ be a bi-Lipschitz mapping. We call $(E,g)$ a \textbf{Hard Sard pair for $f$} if there is a constant $C_{Lip}$ such that the following conditions hold.

Write $\RR^{n+m}=\RR^n \times \RR^m$ in the standard way, and points of $\RR^{n+m}$ as $(x,y)$ with $x\in \RR^n$ and $y\in \RR^m$. Let $F = f \circ g^{-1}$.

We ask that:
\begin{enumerate}[(i)]
\item\label{HS2} $g$ extends to a globally $C_{Lip}$-bi-Lipschitz homeomorphism from $\RR^{n+m}$ to $\RR^{n+m}$.
\item\label{HS3} If $(x,y)$ and $(x',y')$ are in $g(E)$, then $F(x,y) = F(x',y')$ if and only if $x=x'$. Equivalently,
$$ F^{-1}(F(x,y)) \cap g(E)) = (\{x\} \times \RR^m) \cap g(E)$$
\item\label{HS4} The map
$$(x,y) \mapsto (F(x,y),y)$$
is $C_{Lip}$-bi-Lipschitz on the set $g(E)$. In particular, for all $y\in \RR^m$, the restriction
$$F|_{(\RR^n\times\{y\}) \cap g(E)}$$
is $C_{Lip}$-bi-Lipschitz.
\end{enumerate}

If $E\subseteq Q_0$ is a set and there exists a mapping $g\colon \RR^{n+m}\rightarrow \RR^{n+m}$ satisfying \eqref{HS2}-\eqref{HS4} for $E$, then we call $E$ a \textbf{Hard Sard set for $f$}.
\end{definition}
We think of $g$ as a globally bi-Lipshitz change of coordinates that ``straightens out'' the fibers of $f|_E$. Conditions \eqref{HS3} and \eqref{HS4} say, to quote \cite{AS12}, that ``inside $g(E)$, $F$ is independent of $y$, and for fixed $y$, the function $F$ is bi-Lipschitz in $x$.'' 

Observe that the linear projection mapping $\pi(x,y)=x$ on $Q_0$ satisfies all the properties requested for the map $F=f\circ g^{-1}$ on $g(E)$ in Definition \ref{def:HSpair}. Thus, we interpret Definition \ref{def:HSpair} as saying that $f$ ``looks like a projection'' when restricted to $E$, up to globally a bi-Lipschitz change of coordinates $g$.

As a small note, it is clear that if $(E,g)$ is a Hard Sard pair for a Lipschitz map $f$, then so is $(\overline{E},g)$, so a Hard Sard set can always be taken compact.

Azzam and Schul prove the following in \cite{AS12}.
\begin{theorem}[Theorem I of \cite{AS12}]\label{thm:HS}
Let $Q_0$ be the unit cube of $\RR^{n+m}$. Suppose that $f\colon Q_0 \rightarrow X$ is a $1$-Lipschitz function into a metric space,
$$ 0 < \HH^n(f(Q_0)) \leq 1,$$
and
$$ 0 < \delta \leq \HH^{n,m}_\infty(f,Q_0).$$

Then there are constants $C_{Lip}>1$ and $\eta>0$, depending only on $n$, $m$, and $\delta$, such that there is a Hard Sard pair $(E,g)$ for $f$ in $Q_0$ with
$$ \HH^{n+m}(E) \geq \eta > 0.$$
\end{theorem}

\begin{remark}
Due to our slightly stronger Definition \ref{def:HSpair}, the version of Theorem \ref{thm:HS} stated above is not quite what is stated in \cite{AS12}, though it is what is proven there. See Remark \ref{rmk:defdifferences} for further discussion.
\end{remark}

As explained in \cite{AS12}, Theorem \ref{thm:HS} can be viewed as a ``quantitative implicit function theorem'' for Lipschitz maps into metric spaces. In rough terms, Theorem \ref{thm:HS} says that if $\HH^{n,m}_\infty(f,Q_0)>0$ and $0<\HH^n(f(Q_0))\leq 1$, then there is a large set on which $f$ looks like a projection (up to globally bi-Lipschitz change of coordinates). In \cite[Corollary 1.4]{AS12} (see also Lemma \ref{lem:HScontent} below), Azzam and Schul observe that the condition $\HH^{n,m}_\infty(f,Q_0)>0$ is also necessary, in a quantitative sense, for the conclusion of the theorem to hold. (But see Remark \ref{rmk:defdifferences} for some discussion of this point.)

We note that a ``qualitative'' version of Theorem \ref{thm:HS} was recently proven by Haj\l asz-Zimmerman in \cite{HZ}.

The condition $\HH^{n,m}_\infty(f,Q_0)>0$ appearing in Theorem \ref{thm:HS} is quite subtle, as the following result of Kaufman shows.
\begin{theorem}[Kaufman \cite{Ka}]\label{thm:Kaufman}
There is a surjective $C^1$ mapping $g\colon [0,1]^3\rightarrow [0,1]^2$ whose derivative has rank $1$ everywhere. 
In particular, $\HH^{2,1}_\infty(g, [0,1]^3)=0$ even though $\HH^{2}(g([0,1]^3))>0$.
\end{theorem}
The second statement in Theorem \ref{thm:Kaufman} follows from Proposition \ref{prop:Jacobian} below and was not part of \cite{Ka}. 
Note that the mapping $g$ in Theorem \ref{thm:Kaufman} could not be $C^2$, by Sard's theorem. Further discussion of Kaufman's theorem and its relatives appears in subsection \ref{subsec:questions} below.

\subsection{New Results}

After proving Theorem \ref{thm:HS}, Azzam and Schul asked if the result could be pushed further: Beyond guaranteeing a set of a certain size on which the map looks like a projection, can we provide a quantitative decomposition of the domain $Q_0$ into finitely many pieces on which the map looks like a projection, combined with a garbage set which is small in some sense? Observe that Theorem \ref{thm:lipbilip} has this property: it not only guarantees a single nice set where good behavior happens, it guarantees a quantitative exhaustion of the domain by such sets, up to a small garbage set.

To be more specific, Azzam and Schul asked the following:

\begin{question}[\cite{AS12}, Remark 6.15]\label{q:hardersarder}
Let $Q_0$ be the unit cube in $\RR^{n+m}$ and let $f\colon Q_0\rightarrow X$ be a $1$-Lipschitz map into a metric space $X$ with
$$ \HH^n(f(Q_0)) \leq 1.$$
Given $\gamma>0$, do there exist constants $M$ and $C_{Lip}$, depending only on $n$, $m$, and $\gamma$, and Hard Sard pairs $(E_1,g_1), \dots, (E_M,g_M)$ such that
\begin{equation}\label{eq:garbageset}
 \HH^{n,m}_\infty(f, Q_0\setminus \cup E_i) < \gamma ?
\end{equation}
\end{question}
A more tempting requirement, to replace \eqref{eq:garbageset}, would be to ask that
$$ \HH^{n}_\infty(f(Q_0\setminus \cup E_i)) < \gamma,$$
but Kaufman's example in Theorem \ref{thm:Kaufman} shows that this is not in general achievable.

Here, we answer Question \ref{q:hardersarder} in the affirmative. This gives a quantitative decomposition of Lipschitz mappings into a general metric space, up to set with small ``mapping content''. The following is the main result of this paper.

\begin{maintheorem}\label{thm:hardersarder}
Let $Q_0$ be the unit cube in $\RR^{n+m}$ and let $f\colon Q_0\rightarrow X$ be a $1$-Lipschitz map into a metric space $X$ with
$$ \HH^n(f(Q_0)) \leq 1.$$

Given any $\gamma>0$, we can write
$$ Q_0 = E_1 \cup \dots \cup E_M \cup G,$$
where $E_i$ are Hard Sard sets and
$$ \HH^{n,m}_\infty(f,G) < \gamma.$$
The constant $M$ and the constants $C_{Lip}$ associated to the Hard Sard pairs $(E_i,g_i)$ depend only on $n$, $m$, and $\gamma$.
\end{maintheorem}

Theorem \ref{thm:hardersarder} implies Theorem \ref{thm:HS}, but a number of the techniques used by Azzam-Schul in \cite{AS12} also appear in some form in our proof. Theorem \ref{thm:hardersarder} is new even in the case $X=\RR^n$.

\begin{remark}\label{rmk:necessary}
It is natural to ask whether the assumptions in Theorem \ref{thm:hardersarder} are necessary. In Lemma \ref{lem:HScontent}, which is a slight reworking of \cite[Corollary 1.4]{AS12}, we observe that the condition $\HH^{n,m}_\infty(f, Q_0)>0$ is quantitatively necessary in order for $f$ to admit any Hard Sard pair.

With a bit more effort, in Proposition \ref{prop:example}, we also show by an explicit construction that the condition $\HH^n(f(Q_0))\leq 1$ is also necessary for Theorem \ref{thm:hardersarder} to hold (with quantitative control).
\end{remark}

A natural question, given Theorems \ref{thm:HS}, \ref{thm:hardersarder}, and \ref{thm:Kaufman}, is: Under what conditions can one guarantee positivity of $\HH^{n,m}_\infty(f,Q_0)$, and hence a Hard Sard set for $f$? Theorem \ref{thm:Kaufman}, as well as the the construction in Section 2 of \cite{AS12}, indicates that this is not an easy question, and we discuss it further in subsection \ref{subsec:questions}. We give a simple quantitative condition in the case $n=1$:

\begin{maintheorem}\label{thm:onedim}
Let $m$ be a non-negative integer and $\alpha>0$. Then there is an $\eta>0$, depending only on $m$ and $\alpha$, with the following property:
If $f:[0,1]^{1+m}\rightarrow X$ is $1$-Lipschitz and $\diam f([0,1]^{1+m})\geq \alpha$, then $\HH^{1,m}_\infty(f,[0,1]^{1+m})\geq \eta.$ 
\end{maintheorem}
In particular, suppose that $f$ satisfies the assumptions of Theorem \ref{thm:onedim} and $\HH^{1}(f([0,1]^{1+m}))\leq 1$. Then, using either Theorem \ref{thm:HS} or Theorem \ref{thm:hardersarder}, $f$ is guaranteed a Hard Sard set whose $\HH^{1+m}$-measure is bounded below depending only on the diameter of the image of $f$.

\begin{remark}\label{rmk:defdifferences}
We take this opportunity to remark on some small differences between the statements of \cite{AS12} and our statements.

Definition \ref{def:HSpair} actually contains two strengthenings of the conclusions in \cite[Theorem I]{AS12}.

First of all, the analog of Condition \eqref{HS3} in \cite[Theorem I]{AS12} states only a containment
$$ F^{-1}(F(x,y)) \cap g(E)) \subseteq (\{x\} \times \RR^m) \cap g(E),$$
where as Definition \ref{def:HSpair} requires an equality of these fibers.

Second of all, the analog of Condition \eqref{HS4} in \cite[Theorem I]{AS12} contains only the statement that each restriction
$$F|_{(\RR^n\times\{y\}) \cap g(E)}$$
is $C_{Lip}$-bi-Lipschitz, and not the stronger condition that
$$(x,y) \mapsto (F(x,y),y)$$
is $C_{Lip}$-bi-Lipschitz on the set $g(E)$.

Thus, our statement of Theorem \ref{thm:HS}, which is \cite[Theorem I]{AS12}, is slightly stronger than the one given in \cite{AS12}. In fact, the stronger version is actually achieved by the proof in \cite{AS12}, though not stated explicitly. In any case, our proof of Theorem \ref{thm:hardersarder} will use the stronger version of Definition \ref{def:HSpair} given above, and does not rely on Theorem \ref{thm:HS}.

The importance of the strengthened Definition \ref{def:HSpair} lies in Lemma \ref{lem:HScontent} and Proposition \ref{prop:example}. (In fact, the proof of the analog of Lemma \ref{lem:HScontent} in \cite[Corollary 1.4]{AS12} is incorrect as written, and actually requires our strengthened Condition \eqref{HS4} of Definition \ref{def:HSpair}.)
\end{remark}

\subsection{The two main propositions in the proof of Theorem \ref{thm:hardersarder}}
Theorem \ref{thm:hardersarder} cannot be proven by a naive iteration of Theorem \ref{thm:HS}. Theorem \ref{thm:HS} is a statement about the structure of maps on the unit cube (and its proof heavily reflects this). If one uses it to generate a large Hard Sard set $E_1\subseteq Q_0$, one cannot apply it again to generate another large Hard Sard set $E_2$ in the complement of $E_1$. Of course, one can locate small cubes in the complement of $E_1$ and apply a suitably rescaled Theorem \ref{thm:HS} on each of those, but then one apparently gives up all quantitative control on the size and number of these sets.

Thus, the proof of Theorem \ref{thm:hardersarder} occupies the main part of the paper, from Section \ref{sec:directional} through Section \ref{sec:mainproof}. The theorem essentially follows from two main preliminary results that may be of independent interest, and which we now describe.

For the remainder of this discussion, we fix the unit cube $Q_0\subseteq \RR^{n+m}$. As in Definition \ref{def:HSpair}, we write $\RR^{n+m}=\RR^n\times\RR^m$ and points of $\RR^{n+m}$ as $(x,y)$ where $x\in\RR^n$ and $y\in\RR^m$.

The first main step in the proof of Theorem \ref{thm:hardersarder} is to show that finding Hard Sard sets for a mapping $f$ can be reduced to a different problem: that of ``supplementing'' $f$ by a linear projection in a way that yields a bi-Lipschitz map. 

\begin{mainproposition}\label{prop:directional}

Let $F\subseteq Q_0$ be a Borel set and $f:F\rightarrow X$ a $1$-Lipschitz mapping into a metric space. Let $\alpha>0$ and $L\geq 1$.

Assume that 
$$\HH^n(f(\overline{F})) \leq 1$$
and that the map $h:F\rightarrow X \times [0,1]^m$ defined by 
\begin{equation}\label{eq:hdef}
h(x,y) = (f(x,y),y)
\end{equation}
is $L$-bi-Lipschitz on $F$.

Then we can write
$$ F = E_1 \cup \dots \cup E_M \cup G,$$
where $E_i$ are Hard Sard sets for $f$ and
$$ \HH^{n+m}(G) < \alpha.$$

Moreover, the bi-Lipschitz maps $g_i$ associated to each Hard Sard set $E_i$ in this decomposition are ``shears'' on $E_i$, in the sense that there are Lipschitz maps $\psi_i:E_i \rightarrow \RR^n$ such that
\begin{equation}\label{eq:gshear}
 g_i(x,y) = (\psi_i(x,y),y) \text{ for all } (x,y)\in E_i.
\end{equation}

The constant $M$ and the constants $C_{Lip}$ associated to the Hard Sard sets $E_i$ depend only on $n$, $m$, $L$, and $\alpha$.
\end{mainproposition}
We note that Proposition \ref{prop:directional} \textit{can} be proven by iterating certain arguments in \cite{AS12} and keeping careful track of the constants.

As a consequence of Proposition \ref{prop:directional}, proving Theorem \ref{thm:hardersarder} boils down to finding sets on which $f$ can be supplemented by a linear projection to yield a bi-Lipschitz map. It is not easy to find such sets directly, but we show that if one is willing to pre-compose $f$ by bi-Lipschitz mappings, then this can be arranged. This is the content of the second main step in the proof of Theorem \ref{thm:hardersarder}:

\begin{mainproposition}\label{prop:hbilip}
Let $f\colon Q_0 \rightarrow X$ be a $1$-Lipschitz mapping into a metric space $X$.

For each $\alpha>0$, there is a decomposition of $Q_0$ into Borel sets
$$ Q_0 = F_1 \cup F_2 \cup \dots \cup F_N \cup G$$
with the following properties:
\begin{itemize}
\item For each $i\in\{1,\dots,N\}$, there is a bi-Lipschitz map $\phi_i\colon F_i\rightarrow Q_0$ 
such that the map
\begin{equation}\label{eq:propmap}
 (x,y) \mapsto (f(\phi_i^{-1}(x,y)), y),
\end{equation}
is bi-lipschitz on $\phi_i(F_i)$, and 
\item $\HH^{n,m}_\infty(f,G) < \alpha.$
\end{itemize}
The number of sets $N$ and the bi-Lipschitz constants for $\phi_i$ and for the mappings in \eqref{eq:propmap} depend only on $\alpha$,$n$, and $m$.
\end{mainproposition}

\begin{remark}
Unlike Theorem \ref{thm:hardersarder} and Proposition \ref{prop:directional}, Proposition \ref{prop:hbilip} does not require any assumption on the $\HH^n$-measure of the image of $f$, and thus has potentially wider applicability. See, for example, Corollary \ref{cor:arbcontent}.
\end{remark}


\begin{remark}
In an earlier version of this paper, we overlooked that Proposition \ref{prop:hbilip} is closely related to ideas of David and Semmes in \cite{DSregular}. David and Semmes study a sub-class of Lipschitz mappings between Euclidean spaces known as ``$(s,t)$-regular mappings''. In \cite[Theorem 6.1]{DSregular}, they prove that such a mapping admits a single large piece on which it can be supplemented by a linear mapping to become bi-Lipschitz. Furthermore, in \cite[Section 10]{DSregular}, they give a somewhat informal outline of a method to supplement such mappings by more general ``weakly Lipschitz'' mappings to obtain a ``weakly bi-Lipschitz'' mapping on the whole domain. (We do not use this terminology, so we defer the definitions to that paper.) This is closely related to the decomposition result in Proposition \ref{prop:hbilip}, and implies something quite similar for regular mappings.

Our Proposition \ref{prop:hbilip} applies to general Lipschitz mappings (not only regular mappings), which requires bringing in the notion of mapping content, and it allows the target of the mapping to be an arbitrary metric space, rather than a Euclidean space. Proposition \ref{prop:hbilip} can therefore be viewed as simultaneously extending to a new setting and filling in all the details of the informal outline in \cite[Section 10]{DSregular}. (For the regular mappings studied in \cite{DSregular}, the mapping content is not necessary and the garbage set can be controlled by simpler quantities.)

Our proof is also somewhat different; in particular, we do not use Carleson's Corona construction. That said, the reader will certainly notice many similarities between our approach to Proposition \ref{prop:hbilip} and the ideas in \cite[Section 10]{DSregular}, and it was a serious oversight on our part to overlook this.
\end{remark}

It should now not be difficult to believe that Propositions \ref{prop:directional} and \ref{prop:hbilip} combine to prove Theorem \ref{thm:hardersarder}, and we provide the details in Section \ref{sec:mainproof}.

\subsection{Additional questions and relation to quantitative topology}\label{subsec:questions}
We conclude this introduction by connecting our results to some recent developments in quantitative topology, and stating a few questions. 

\subsubsection{When is mapping content small?}
Many of our questions are specific instances of the following general question: 
\begin{question}\label{q:general}
What can be said about the map $f$ if $\HH^{n,m}_\infty(f,Q_0)=0$? If $\HH^{n,m}_\infty(f,Q_0)$ is small?
\end{question}
Let us discuss the status of this question in the first few cases:

\begin{itemize}
\item If $n=0$ in Question \ref{q:general}, then $\HH^{n,m}_\infty(f,Q_0)$ is simply (comparable to) the $m$-dimensional Hausdorff content of $Q_0$, and thus can never be small.

\item If $m=0$, then $\HH^{n,m}_\infty(f,Q_0) \approx \HH^n_\infty(f(Q_0))$ (see Lemma \ref{lem:areaboundscontent}), so Question \ref{q:general} can be answered by saying that the mapping has small image in dimension $n$.

\item If $n=1$ in Question \ref{q:general}, then Theorem \ref{thm:onedim} implies that $\diam(f(Q_0))$ is zero or small, i.e., that $f$ must be constant or near constant. This completely answers Question \ref{q:general} in that case.
\end{itemize}
Thus, the first unanswered case of Question \ref{q:general} is the case $n=2$, $m=1$. In this scenario, Kaufman's construction (Theorem \ref{thm:Kaufman}) already shows that no such simple statement is available, as it yields a highly non-trivial mapping $g$ with $\HH^{2,1}_\infty(g,Q_0)=0$. A class of related examples, mapping into general metric spaces, was given is discussed in  \cite[Section 2.2]{AS12}. Kaufman's example and Azzam-Schul's examples have the property that the mappings involved \textit{factor through trees}.

For us, a \textit{tree} is a compact, geodesic metric space $T$ such that every two points in $T$ are the endpoints of a unique arc in $T$. We say that a Lipschitz mapping $f\colon Q_0 \rightarrow X$ \textit{factors through a tree} if there is a tree $T$ and Lipschitz maps $g\colon Q_0 \rightarrow T$ and $h\colon T \rightarrow X$ such that $f = h\circ g$. (A number of recent papers consider this notion; see \cite{WY, Zust}.)

For the first unanswered case of Question \ref{q:general} (the case $n=2$ and $m=1$), we pose the following conjecture.

\begin{conjecture}\label{conj:21}
Let $Q_0=[0,1]^3$ and let $f\colon Q_0 \rightarrow X$ be a $1$-Lipschitz mapping into a metric space. Assume without loss of generality that $X\subseteq \ell_\infty$.
\begin{itemize}
\item (Qualitative version) If $\HH^{2,1}_\infty(f,Q_0)=0$, then $f$ factors through a tree.

\item (Quantitative version) For every $\epsilon>0$, there is a $\delta=\delta(\epsilon)$ with the following property: If $\HH^{2,1}_\infty(f,Q_0)<\delta$, then there is a $1$-Lipschitz map $g\colon Q_0 \rightarrow \ell_\infty$ that factors through a tree and satisfies $\|g-f\|_\infty<\epsilon$.
\end{itemize}

\end{conjecture}
As far as we know, the analog of Conjecture \ref{conj:21} may hold even if $n=2$ and $m$ is allowed to be arbitrary. The techniques in \cite{WY} may be relevant here. 

For still higher dimensions $n$, one might hope that Question \ref{q:general} could be answered by reference to some type of quantitative topological non-degeneracy of the mapping in dimension $n$. Here the picture appears to be significantly more complicated and it is unclear (at least to the present authors) what conjectures to make. The following theorem of Wenger-Young indicates this. 

\begin{theorem}[\cite{WY}, Theorem 2]\label{thm:WY}
Let $n\leq n+m-1 < 2n-3$. Then any Lipschitz map
$$ f: \partial [0,1]^{n+m} \rightarrow \partial [0,1]^n$$
can be extended to a Lipschitz map
$$ \hat{f}: [0,1]^{n+m} \rightarrow \RR^n$$
whose derivative has rank $\leq n-1$ almost everywhere. 
\end{theorem}
(We have modified the original statement slightly to fit our framework: Given $n,m$, our statement corresponds to setting Wenger-Young's $n$ as our $n-1$, their $k$ as our $n-1+m$, and making the usual Lipschitz identifications between balls and cubes.) As a concrete choice in Theorem \ref{thm:WY}, one may take $n=4$ and $m=1$.

In particular, the map $\hat{f}$ in Theorem \ref{thm:WY} will always have $\HH^{n,m}_\infty(f,[0,1]^{n+m})=0$ (by Proposition \ref{prop:Jacobian} below), but if $f$ is chosen correctly then $\hat{f}$ will not be able to factor through a tree for topological reasons. For further discussion of the topology behind Theorem \ref{thm:WY}, and more recent developments, we refer the reader to \cite{WY, GHP}.

\subsubsection{Alternative versions of mapping content}

In the definitions of Hausdorff content and Hausdorff measure, it does not much matter whether one allows covers by balls, as we do above, or by dyadic cubes, or by arbitrary sets. This affects the definition only up to dimensional constants, as an easy computation shows.

One can ask the same question about mapping content $\HH^{n,m}_\infty$: could one get an equivalent quantity allowing covers by arbitrary sets rather than only dyadic cubes? To be more specific, if $f\colon Q_0 \rightarrow X$ is a mapping and $A\subseteq Q_0$, let
\begin{equation}\label{eq:arbdef}
\hat{\HH}^{n,m}_\infty(f,A) = \inf \sum_i \HH^{n}_\infty(f(S_i))\diam(S_i)^m,
\end{equation}
where the infimum is over all countable covers $\{S_i\}$ of $A$ by \textit{arbitrary} subsets of $Q_0$. It is clear that
\begin{equation}\label{eq:arbcontent}
\hat{\HH}^{n,m}_\infty(f,A) \lesssim_{n,m} \HH^{n,m}_\infty(f,A),
\end{equation}
because dyadic cubes $Q_i$ are admissible sets, and the diameter of a dyadic cube is comparable to its side length. (In particular, the statement of Theorem \ref{thm:hardersarder} is \textit{a priori} stronger for using $\HH^{n,m}_\infty$ than it would be if it used $\hat{\HH}^{n,m}_\infty$.)

However, a similar bound in the reverse direction does not seem to be easy to show. If one is given a good cover of $A$ by arbitrary sets $S_i$ that almost achieve the infimum in \eqref{eq:arbdef}, then one can certainly cover each $S_i$ by dyadic cubes $Q_{i,j}$ to get an admissible cover for $\HH^{n,m}_\infty(f,A)$. However, since the cubes $Q_{i,j}$ may contain points outside of $\cup_i S_i$, and the images $f(Q_{i,j})$ may overlap, it does not seem to be clear how to control
$$\sum \HH^{n,m}_\infty(f(Q_{i,j})) \side(Q_{i,j})^m$$
by the sum over $S_i$ in \eqref{eq:arbdef}.

As a short \textit{corollary} of one of our main results (Proposition \ref{prop:hbilip}), we show that the two versions of mapping content are related in a weak sense for $1$-Lipschitz mappings into metric spaces. In particular, they vanish simultaneously, with some quantitative control. 

Fix $n$ and $m$, and let $Q_0$ denote the unit cube of $\RR^{n+m}$.

\begin{maincor}\label{cor:arbcontent}
For each $\delta>0$, there is a $\delta'>0$ with the following property:

If $f:Q_0\rightarrow X$ is a $1$-Lipschitz mapping into a metric space, and $A\subseteq Q_0$ has
$$ \HH^{n,m}_\infty(f,A) \geq \delta,$$
then
$$ \hat{\HH}^{n,m}_\infty(f,A) \geq \delta'.$$
The number $\delta'$ depends only on $\delta$, $n$, and $m$.
\end{maincor}
We prove Corollary \ref{cor:arbcontent} in Section \ref{sec:arbcontent}. As observed above, we do not know if there is a simple direct argument that yields Corollary \ref{cor:arbcontent}, independently of the results of this paper.

It is natural to ask if there is in fact a linear relationship between these two quantities:
\begin{question}
Is there a constant $c=c_{n,m}$ such that
$$ \hat{\HH}^{n,m}_\infty(f,A) \geq c\HH^{n,m}_\infty(f,A)$$
for all $1$-Lipschitz maps $f:Q_0\rightarrow X$ into a metric space and all $A\subseteq Q_0$?
\end{question}
One could also imagine alternative versions based on covers by balls or other families of sets, and the same types of questions would apply.

\subsection{Structure of the paper}
In Section \ref{sec:prelims}, we give the basic definitions and notations in the paper. We also state some necessary theorems from \cite{AS12} and \cite{AS14} that underlie the proof of Theorem \ref{thm:hardersarder}.

In Section \ref{sec:mappingcontent}, we outline some basic properties of the ``mapping content'' $\HH^{n,m}_\infty$. 

Section \ref{sec:directional} contains the proof of Proposition \ref{prop:directional}, which builds on ideas from \cite[Section 6]{AS12} with an additional iteration scheme.

Sections \ref{sec:codesplit} and \ref{sec:compressed} provide some preliminary results needed in the proof of Proposition \ref{prop:hbilip}. The former gives two lemmas allowing the coding and splitting of families of cubes into families with various useful properties, and the latter shows how to control the mapping content of a certain family of cubes that will form part of the ``garbage set'' in Proposition \ref{prop:hbilip}.

We then prove Proposition \ref{prop:hbilip} in Section \ref{sec:hbilip}, and combine Propositions \ref{prop:directional} and \ref{prop:hbilip} to prove Theorem \ref{thm:hardersarder} in Section \ref{sec:mainproof}.

Section \ref{sec:onedim} contains the proof of Theorem \ref{thm:onedim}, which is essentially independent from the rest of the paper.

Finally, Section \ref{sec:example} contains an explicit construction (mentioned in Remark \ref{rmk:necessary}) showing that the assumption $\HH^n(f(Q_0))\leq 1$ is quantitatively necessary for Theorem \ref{thm:hardersarder} to hold, and Section \ref{sec:arbcontent} contains the proof of Corollary \ref{cor:arbcontent}.

\section{Preliminaries}\label{sec:prelims}

This section contains notation, definitions and a few fundamental results used in the rest of the paper.

\subsection{Basic metric space and mapping notions}
Throughout most of the paper, $(X,d)$ will be an arbitrary metric space, usually written $X$ if the metric is understood. If $(X,d_X)$ and $(Y,d_Y)$ are metric spaces, there are a number of natural, bi-Lipschitz equivalent metrics to put on the product $X\times Y$. For convenience, unless otherwise noted, we will equip $X\times Y$ with the metric
\begin{equation}\label{eq:productmetric}
 d_{X\times Y}((x,y),(x',y')) = \max\{ d_X(x,x'), d_Y(y,y')\}.
\end{equation}
An exception to this rule is when we write $\RR^{n+m}$ as $\RR^n \times \RR^m$, in which case we continue to equip it with the standard Euclidean metric.

We will use the notation
$$ \pi_X : X\times Y \rightarrow X \text{ and } \pi_Y : X\times Y \rightarrow Y$$
to denote the projections mapping $(x,y)\in X\times Y$ to $x$ and $y$, respectively.

If $E$ is a subset of a metric space $X$, we write
$$ \diam(E) = \sup\{d(x,y): x,y\in E\}.$$
If $x\in X$, we write
$$\dist(x,E) = \inf\{d(x,y):y\in E\}.$$
If $\delta>0$, the \textit{$\delta$-neighborhood} of $E$ in $X$ is
$$ N_\delta(E) = \{x\in X: \dist(x,E) < \delta\}.$$

A mapping $f$ from a metric space $(X,d_X)$ to a metric space $(Y,d_Y)$ is called \textit{Lipschitz} (or \textit{$C$-Lipschitz} to emphasize the constant) if there is a constant $C$ such that
$$ d_Y(f(x),f(y)) \leq Cd(x,y) \text{ for all } x,y\in X.$$
The mapping $F$ is called \textit{bi-Lipschitz} (or \textit{$C$-bi-Lipschitz}) if
$$ C^{-1}d_X(x,y) \leq  d_Y(f(x),f(y)) \leq Cd_X(x,y) \text{ for all } x,y\in X.$$

\subsection{Hausdorff measure and Hausdorff content}\label{subsec:hausdorff}
For a metric space $(X,d)$ and a (not necessarily integer) constant $k\geq 0$, the \textit{$k$-dimensional Hausdorff measure} $\HH^k(A)$ of a subset $A\subseteq X$ is 
$$ \HH^k(A) = \lim_{\delta\rightarrow 0} \inf\sum_{B\in\mathcal{B}_\delta} \diam(B)^k,$$
where the infimum is taken over all covers $\mathcal{B}_\delta$ of $A$ by closed balls of diameter at most $\delta$. In $\RR^k$, it is standard that $\HH^k$ is comparable to the $k$-dimensional Lebesgue measure.

The $k$-dimensional Hausdorff content $\HH^k_\infty(A)$ is defined similarly, but without the restriction to small diameters:
$$ \HH^k_\infty(A) = \inf\sum_{B\in\mathcal{B}} \diam(B)^k,$$
where the infimum is taken over all covers $\mathcal{B}$ of $A$ by closed balls. It is easy to see that $\HH^k_\infty$ is countably sub-additive, but not in general a measure.

In general, one always has the trivial inequality $\HH^k_\infty(A) \leq \HH^k(A)$. Moreover, the two quantities vanish simultaneously: $\HH^k(A)=0$ if and only if $\HH^k_\infty(A)=0$ \cite[Exercise 8.6]{He}. Lastly, in $\RR^d$, it is a standard fact that the $d$-dimensional Hausdorff content and measure are always comparable:

\begin{lemma}\label{lem:contentmeasure}
For each $d\geq 1$, there is a constant $C_d$ such that
$$ \HH^d_\infty(A) \leq \HH^d(A) \leq C_d\HH^d_\infty(A)$$
for all subsets $A\subseteq \RR^d$.
\end{lemma}
\begin{proof}
The first inequality is immediate from the definitions, as noted above. For the second, fix $\epsilon>0$ arbitrary. Let $\mathcal{B}$ be a cover of $A$ by closed balls such that
$$ \sum_{B\in\mathcal{B}} \diam(B)^d \leq  \HH^d_\infty(A) +\epsilon.$$
Let $\mathcal{B}'\subseteq \mathcal{B}$ be a disjoint subcollection such that 
$$ A\subseteq \cup_{B\in\mathcal{B}'} (5B).$$
(See, e.g., \cite[Theorem 1.2]{He}.)
Then
$$ \HH^d(A) \leq \sum_{B\in\mathcal{B}'} \HH^d(5B) \lesssim 5^d \sum_{B\in\mathcal{B}'}\diam(B)^d \lesssim \HH^d_\infty(A) + \epsilon.$$
Sending $\epsilon$ to $0$ completes the proof.
\end{proof}

We will also occasionally refer to Lebesgue measure on $\RR^{n+m}$, rather than Hausdorff measure $\HH^{n+m}$. Lebesgue measure will be denoted simply by $|\cdot|$. Since we have not bothered with normalization constants in the definition of $\HH^{n+m}$, these measures are comparable rather than equal.

\subsection{Grassmannians and Hausdorff distance}\label{subsec:grassmannian}
We write $\Gr(k,d)$ to denote the appropriate Grassmannian: the space of $k$-dimensional vector subspaces of $\RR^d$. We will often refer to elements of $\Gr(k,d)$ as ``$k$-planes in $\RR^d$''.

Occasionally, it will be useful to have a metric on $\Gr(k,d)$. In general, the \text{Hausdorff distance} between subsets $A,B$ of a metric space $X$ is defined by
$$ d_{\text{Hausdorff}}(A,B) = \inf\{ \epsilon>0 : \dist(a,B) < \epsilon \text{ and } \dist(b,A)<\epsilon \text{ for all } a\in A, b\in B\}.$$
This is well-known to be a metric on the compact subsets of $X$.

We will occasionally use this to define a metric $D$ on $\Gr(k,d)$ as
$$ D(P,Q) := d_{\text{Hausdorff}}(P\cap \overline{B}(0,1), Q\cap\overline{B}(0,1)).$$
The space $\Gr(k,d)$ is compact with this metric.

\subsection{Dyadic cubes}

In a fixed $\RR^d$, with $d$ generally understood from context, we write $Q_0$ for the unit cube, i.e.,
$$ Q_0 = [0,1]^d.$$
We write $\Delta$ for the collection of all dyadic cubes $Q\subseteq Q_0$, and $\Delta_k$ for the collection of those dyadic cubes with side length $2^{-k}$. 

If $Q\in\Delta$, we write $\side(Q)$ for the side-length of $Q$. Thus, $\side(Q)=2^{-k}$ if and only if $Q\in\Delta_k$.

If $Q\in\Delta$ and $C>0$, we write $CQ$ for a cube with the same center but $C$ times the side length. In particular, if $C$ is an odd positive integer, then $CQ$ is a union of $C^d$ distinct cubes of the same side length as $Q$.

Lastly, we occasionally call a collection of cubes ``almost-disjoint'' if they have disjoint interiors. Such collections arise in the definition of $\HH^{n,m}_\infty$.

\subsection{Metric derivatives}\label{subsec:metricderivative}
Let $X$ be a metric space and $f\colon \RR^d \rightarrow X$ a $1$-Lipschitz function. We will use some results and notation from \cite{AS14}, which were in turn inspired by the idea of metric differentiability in \cite{Kirchheim}.

For a cube $Q\subseteq \RR^d$ let
$$\md_f(Q) :=\frac{1}{\side(Q)} \inf _{\| \cdot\|} \sup_{x, y \in Q} \left| d(f(x),f(y)) - \|x-y\|\right|,$$
where the infimum is taken over all seminorms $\|\cdot\|$ on $\RR^d$. If the function $f$ is understood, we will simply write $\md(Q)$.

The quantity $\md_f(Q)$ measures how well the pullback of the distance in $X$ under $f$ can be approximated by a seminorm in $Q$. For metric space valued functions, it serves as a replacement for measuring ``deviation from linearity''.

We will use the following result of \cite{AS14}, which is a quantitative differentiation result for Lipschitz mappings into metric spaces:
\begin{theorem}[\cite{AS14}, Theorem 1.1]\label{thm:quantdiff}
Let $X$ be a metric space and $f\colon \RR^d \rightarrow X$ a $1$-Lipschitz function. Let $\epsilon>0$ and $C_0>0$. Then
$$ \sum \{ |Q| : Q\in\Delta, \md_f(C_0 Q) > \epsilon\} \leq C_{\epsilon,d}.$$
The constant $C_{\epsilon,d}$ depends only on $\epsilon$, $C_0$, and $d$ but not on the space $X$ or the function $f$.
\end{theorem}

Note that Theorem 1.1 in \cite{AS14} is stated only for $C_0=3$, but the version above follows easily.

In the remainder of the paper, we will only apply Theorem \ref{thm:quantdiff} with $d=n+m$ and
\begin{equation}\label{eq:C0def}
C_0 = 10(n+m) 
\end{equation}

Consider $Q_0=[0,1]^{n+m}$, $C_0$ as above, and $f:Q_0 \rightarrow X$ a Lipschitz function into a metric space. Note that standard compactness arguments show that if $f$ is Lipschitz and $Q\in\Delta$, then there is a seminorm that minimizes the infimum in the definition of $\md_f(C_0Q)$. Thus, we write
\begin{equation}\label{eq:normdef}
 \|\cdot\|_{f,Q}
\end{equation}
for a seminorm $\|\cdot\|$ that minimizes the quantity
$$ \sup_{x,y \in C_0Q} \left| |f(x)-f(y)| - \|x-y\| \right|.$$
If the mapping $f$ is understood from context, we may call the seminorm simply $\|\cdot\|_Q$.

A basic fact about these seminorms is the following simple lemma.

\begin{lemma}\label{lem:normbound}
Let $f:Q_0 \rightarrow X$ be a $1$-Lipschitz function into a metric space $X$, and let $Q\in \Delta$. If $\md_f(C_0Q)<\epsilon$ and $|v|<C_0\side(Q)$, then
\begin{equation}\label{eq:normbound}
\|v\|_{f,Q} \leq |v|+C_0\epsilon\side(Q),
\end{equation}
\end{lemma}
\begin{proof}
Consider any point $x\in C_0Q$ for which $x+v$ is also in $C_0Q$. Then, by definition of $\md_f$ and $\|\cdot\|_{f,Q}$, we have
$$ \|v\|_{f,Q} \leq |f(x) - f(x+v)| + C_0\epsilon\side(Q) \leq |v|+C_0\epsilon\side(Q).$$
\end{proof}

\subsection{Bi-Lipschitz extension}
An important step in Azzam-Schul's proof of Theorem \ref{thm:HS} is the following bi-Lipschitz extension result.

\begin{theorem}[Theorem II of \cite{AS12}]\label{thm:ASextension}
Let $D\geq n$ and $\kappa\in (0,1)$. There is a constant $M=M(\kappa, D)$ such that if $f\colon \RR^n \rightarrow \RR^D$ is $1$-Lipschitz, then the following hold:
\begin{enumerate}[(i)]
\item There are sets $E_1, \dots, E_M$ such that 
$$ \HH^n_\infty\left( f\left([0,1]^n \setminus \cup E_i\right)\right) \lesssim_D \kappa.$$
\item For each $i\in\{1,\dots, M\}$, there is an $L$-bi-Lipschitz map $F_i\colon \RR^n \rightarrow \RR^D$, with $L \lesssim_D \frac{1}{\kappa}$, such that
$$ F_i|_{E_i} = f|_{E_i}.$$
\end{enumerate}
\end{theorem}

Theorem \ref{thm:ASextension} allows one to not only find bi-Lipschitz pieces of Lipschitz mappings, but to ensure that those bi-Lipschitz pieces can be globally extended. In fact, we will only use the following immediate consequence of this result.

\begin{corollary}\label{cor:bilipextension}
Let $D\geq n$, $C\geq 1$, and $\kappa\in (0,1)$. There are constants $M=M(\kappa, C, D)$ and $L=L(\kappa,C,D)$ such that if $A\subseteq [0,1]^n$ and  $f\colon A\rightarrow \RR^D$  is  $C$-bi-Lipschitz,
then the following hold:
\begin{enumerate}[(i)]
\item There are sets $E_1, \dots, E_M$ such that 
$$ \HH^n_\infty\left( A \setminus \cup E_i\right) < \kappa.$$
\item For each $i\in\{1,\dots, M\}$, there is an $L$-bi-Lipschitz map $F_i\colon \RR^n \rightarrow \RR^D$ such that
$$ F_i|_{E_i} = f|_{E_i}.$$
\end{enumerate}
\end{corollary}
\begin{proof}
Let $n$, $D$, $C$, $\kappa$, and $f\colon A \rightarrow \RR^D$ be as in the statement of the corollary.

By rescaling, we may assume without loss of generality that $f$ is $1$-Lipschitz and $C$-bi-Lipschitz on $A$. Furthermore, we may extend $f$ to a $1$-Lipschitz mapping from $\RR^n$ to $\RR^D$, by Kirszbraun's theorem.

Apply Theorem \ref{thm:ASextension} to $f$ to obtain sets $E_1, \dots, E_M$ such that
$$ \HH^n_\infty\left( f\left([0,1]^n \setminus \cup E_i\right)\right) < \kappa/C^{n}$$
and $f|_{E_i}$ admits a globally $L$-bi-Lipschitz extension $F_i\colon \RR^n \rightarrow \RR^D$.

In that case, since $f$ is $C$-bi-Lipschitz on $A$, we have
$$
 \HH^n_\infty\left( A \setminus \cup E_i\right) \leq C^{n} \HH^n_\infty\left( f\left([0,1]^n \setminus \cup E_i\right)\right) < \kappa
$$
This completes the proof.
\end{proof}

\section{About mapping content}\label{sec:mappingcontent}
In this section, we summarize some basic properties of the ``mapping content'' $\HH^{n,m}_\infty$. Most of the statements and arguments appear already in \cite{AS12} and \cite{HZ}, but we have included proofs in cases where we have modified the original statements.

For the remainder of this section, let $Q_0=[0,1]^{n+m}\subseteq \RR^{n+m}$ and let $X$ be an arbitrary metric space. The first two observations are simple consequences of the definition.

\begin{lemma}\label{lem:easymapping}
If $f\colon Q_0 \rightarrow X$ is a mapping, then
$$\HH^{n,m}_\infty(f,Q_0) \leq \HH^{n}_\infty(f(Q_0)).$$
\end{lemma}
\begin{proof}
This is immediate from the definition: just use the single cube $Q_0$ in the infimum defining $\HH^{n,m}_\infty(f,Q_0).$
\end{proof}

\begin{lemma}\label{lem:subadditive}
If $f\colon Q_0 \rightarrow X$ is Lipschitz, then $\HH^{n,m}_\infty$ is countably sub-additive on subsets of $Q_0$. In other words, if $\{A_i\}_{i=1}^\infty$ is a countable collection of subsets of $Q_0$, then
$$\HH^{n,m}_\infty(f, \cup_i A_i) \leq \sum_{i=1}^\infty \HH^{n,m}_\infty(f,A_i).$$
\end{lemma}
\begin{proof}
Fix $\epsilon>0$. For each $A_i$, let $\{Q_i^j\}_{j\in J_i}$ denote a collection of almost-disjoint dyadic cubes such that
$$ \sum_{j\in J_i} \HH^{n}_\infty(f(Q_i^j))\side(Q_i^j)^m  \leq \HH^{n,m}_\infty(f,A_i) + \epsilon 2^{-i}.$$
Let 
$$ \mathcal{Q} = \{Q_i^j : i\in\mathbb{N}, j\in J_i\}$$
and let $\mathcal{Q}_0$ be the collection of maximal cubes in $\mathcal{Q}$. Then $\mathcal{Q}_0$ forms an almost-disjoint cover of $\cup A_i$ by dyadic cubes, and so
$$ \HH^{n,m}_\infty(f, \cup_i A_i) \leq \sum_{Q\in\mathcal{Q}_0} \HH^n_\infty(f(Q))\side(Q)^m  \leq \sum_{i\in\mathbb{N}}\sum_{j\in J_i} \HH^{n}_\infty(f(Q_i^j))\side(Q_i^j)^m  \leq \epsilon +  \sum_{i=1}^\infty \HH^{n,m}_\infty(f,A_i).$$
Sending $\epsilon$ to zero completes the proof.
\end{proof}

With more precise information about the pointwise behavior of the mapping, one can get more precise upper bounds for $\HH^{n,m}_\infty$. We do not use the next result (Proposition \ref{prop:Jacobian}) in the remainder of the paper, but we state it to give the reader a better feeling for the quantity $\HH^{n,m}_\infty$. The following notation and result are from \cite{HZ}; the result generalizes \cite[Lemma 6.13]{AS12}.

If $f\colon Q_0 \rightarrow X$ is Lipschitz and $x\in Q_0$, define the quantity
$$ \Theta_*^n(f,x) := \liminf_{r\rightarrow 0} \frac{\HH^n_\infty(f(B(x,r) \cap Q_0))}{r^n}.$$
Note that $\Theta_*^n(f,x)$ is bounded above, with bound depending only on $n$ and the Lipschitz constant of $f$.

\begin{proposition}[\cite{HZ}, Propositions 5.1 and 5.2]\label{prop:Jacobian}
Let $f\colon Q_0 \rightarrow X$ be Lipschitz. We then have:
\begin{enumerate}[(i)]
\item $$ \HH^{n,m}_\infty(f,Q_0) \lesssim \int_Q \Theta_*^n(f,x) \,dx,$$
with implied constants depending only on $n,m$.
\item If $X=\RR^n$, then
 $$  \Theta_*^n(f,x) = |J^n f| := \sqrt{\left|\text{det}\left((Df)(x)\cdot Df(x)^T\right) \right|}$$
for a.e. $x\in Q_0$.
\end{enumerate}
\end{proposition}

We now move to some basic facts about the mapping content of subsets of the cube. First of all, we have the following.
\begin{lemma}\label{lem:areaboundscontent}
If $f:Q_0\rightarrow X$ is $1$-Lipschitz and $A\subseteq Q_0$, then
$$ \HH^{n,m}_\infty(f,A) \lesssim \HH^{n+m}_\infty(A) \leq \HH^{n+m}(A).$$
The implied constant depends only on $n+m$.
\end{lemma}
\begin{proof}
The second inequality is immediate from the definitions of Hausdorff content and Hausdorff measure.

For the first, fix $\epsilon>0$ and consider a cover $\{B_i\}$ of $A$ by balls of radius $r_i$ such that
$$ \sum_i r_i^{n+m} \leq \HH^{n+m}_\infty(A) + \epsilon.$$
Each $B_i$ is contained a union of at most $2^{n+m}$ dyadic cubes $\{Q_i^k\}$ with
$$ \side(Q^k_i) \lesssim r_i.$$

Let $\{R_j\}\subseteq \{Q^k_i\}_{i,k}$ denote the collection of maximal cubes in this collection, which form an almost-disjoint cover of $A$. Then
$$ \HH^{n,m}_\infty(A) \leq \sum_{i,k} \HH^{n}_\infty(f(Q^k_i))\side(Q^k_i)^m \lesssim \sum_{i,k} \side(Q_i^k)^{n+m} \lesssim \sum_i r_i^{n+m} \leq \HH^{n+m}_\infty(A) + \epsilon.$$
Sending $\epsilon$ to zero yields the first inequality.
\end{proof}

Lastly, we point out that the mapping content of a Hard Sard set is comparable to its $(n+m)$-dimensional measure, which shows that the condition $\HH^{n,m}_\infty >0$ is quantitatively necessary to find Hard Sard sets for a mapping $f$. The following is a minor modification of Corollary 1.4 of \cite{AS12}; we include a proof for convenience. The lemma is not used directly in the rest of the paper.

\begin{lemma}\label{lem:HScontent}
If $f:Q_0\rightarrow X$ is $1$-Lipschitz and $E\subseteq Q_0$ is a Hard Sard set for $f$, then
$$ \HH^{n,m}_\infty(f,E) \approx \HH^{n+m}(E),$$
with constants depending only on $n$, $m$, and the Hard Sard constant $C_{Lip}$ of $E$.
\end{lemma}
\begin{proof}
The upper bound
$$ \HH^{n,m}_\infty(f,E) \lesssim \HH^{n+m}(E)$$
follows from the previous lemma. We now focus on the reversed bound.

We write points in $\RR^{n+m}$ as $(x,y)$ with $x\in \RR^n$ and $y\in\RR^m$. We may assume that $\HH^{n+m}(E)>0$, otherwise the inequality we are proving is trivial. 

Fix $\epsilon>0$ arbitrary. Let $g\colon \RR^{n+m} \rightarrow \RR^{n+m}$ be the bi-Lipschitz change of coordinates associated to the Hard Sard set $E$. Let $F = f \circ g^{-1}$. 

Let $\{Q_i\}$ be a cover of $E$ by almost-disjoint dyadic cubes such that
$$ \sum \HH^n_\infty(f(Q_i))\side(Q_i)^m < \HH^{n,m}_\infty(f,E) + \epsilon.$$
Let $A_i = g(Q_i)$ for each $i$, so that the sets $\{A_i\}$ form a cover of $g(E)$ with $F(A_i)=f(Q_i)$. Note that $\diam(A_i) \approx \side(Q_i)$, since $g$ is bi-Lipschitz.

For each $i$, there is a collection of balls $\{B_i^j\}$ covering $F(A_i)$ in $X$ such that
$$ \sum_{j} \diam(B_i^j)^n \leq \HH^n_\infty(F(A_i)) + \epsilon\side(Q_i)^n.$$
(Note that without loss of generality, we may assume that $\diam(B_i^j) \lesssim_{n,m} \diam(F(A_i)) \lesssim \diam(A_i)$ for each $j,i$.)

Therefore,
\begin{align}
\HH^{n,m}_\infty(f, E) &\geq  \sum \HH^n_\infty(f(Q_i))\side(Q_i)^m - \epsilon\\
&= \sum_i \HH^n_\infty(F(A_i))\side(Q_i)^m - \epsilon\\
&\geq \sum_{i,j} \diam(B_i^j)^n\side(Q_i)^m - 2\epsilon \label{eq:maplemmabound}
\end{align}

Let $\tilde{F}$ denote the map
$$ \tilde{F}(x,y) = (F(x,y),y) : g(E) \rightarrow \RR^{n+m}.$$
By Condition \eqref{HS4} of Definition \ref{def:HSpair}, $\tilde{F}$ is bi-Lipschitz with constant $C_{Lip}$. Thus,
$$ \HH^{n+m}_\infty(\tilde{F}(g(E))) \approx \HH^{n+m}_\infty(E).$$

Now, for each fixed $i$, 
$$\tilde{F}(A_i) \subseteq \bigcup_j \left(B_i^j \times \pi_{\RR^m}(A_i)\right). $$
We can cover each $B_i^j \times \pi_{\RR^m}(A_i)$ by 
$$\lesssim \left(\frac{\diam(A_i)}{\diam(B_i^j)}\right)^m \approx \left(\frac{\side(Q_i)}{\diam(B_i^j)}\right)^m$$ 
balls of diameter equal to $\diam(B_i^j)$. Therefore, 
$$ \HH^{n+m}_\infty(\tilde{F}(A_i)) \lesssim \sum_j \left(\frac{\side(Q_i)}{\diam(B_i^j)}\right)^m\diam(B_i^j)^{n+m} = \sum_j \diam(B_i^j)^n \side(Q_i)^m,$$
and so, using \eqref{eq:maplemmabound},
\begin{align*}
 \HH^{n,m}_\infty(f,E) &\geq \sum_{i,j} \diam(B_i^j)^n\side(Q_i)^m -2\epsilon\\
&\geq c_1\sum_i \HH^{n+m}_\infty(\tilde{F}(A_i)) - 2\epsilon\\
&\geq c_2\HH^{n+m}_\infty(\tilde{F}(g(E))) - 2\epsilon \\
&\geq c_3\HH^{n+m}_\infty(E) - 2\epsilon\\
&\geq c_4\HH^{n+m}(E) - 2\epsilon,
\end{align*}
where $c_1, c_2, c_3, c_4$ depend only on $n$, $m$, and the Hard Sard constant $C_{Lip}$ of $(E,g)$. Sending $\epsilon$ to zero completes the proof.

\end{proof}

\section{Bi-Lipschitz supplements yield Hard Sard sets}\label{sec:directional}
In this section, we prove Proposition \ref{prop:directional}. Recall that, informally, this says that if one can find a set $F$ on which the map
$$ h(x,y) = (f(x,y),y)$$
is bi-Lipschitz, then one can quantitatively decompose $F$ further into Hard Sard sets for $f$. We refer to the statement of Proposition \ref{prop:directional} above for the precise assumptions and conclusions.

For the remainder of this section, we fix the unit cube $Q_0\subseteq \RR^{n+m}$. As in Definition \ref{def:HSpair}, we write $\RR^{n+m}=\RR^n\times\RR^m$ and points of $\RR^{n+m}$ as $(x,y)$ where $x\in\RR^n$ and $y\in\RR^m$.

For $y\in\RR^m$, we write $L_y = [0,1]^n\times\{y\}$. If $E\subseteq Q_0$, we write $E_y = E \cap L_y$.

We will need the following lemma, which is \cite[Lemma 6.12]{AS12}. As we state it slightly differently, we include a proof for convenience.
\begin{lemma}[\cite{AS12}, Lemma 6.12]\label{lem:612}
Let $F\subseteq Q_0$ be a Borel set and $f:F\rightarrow X$ a $1$-Lipschitz mapping into a metric space.

Let $E\subseteq F$ be a Borel subset, with $\HH^{n+m}(E)>0$, on which the mapping $h$ defined in \eqref{eq:hdef} is $L$-bi-Lipschitz, for $L\geq 1$.

Then there is a $y'\in \RR^m$ such that
\begin{equation}\label{eq:612}
\int_{[0,1]^m} \HH^n( f(E_{y'}) \cap f(E_z) ) d\HH^m(z) \gtrsim_{n,m} L^{-n} \frac{\HH^{n+m}(E)^2}{\HH^n(f(\overline{E}))}
\end{equation}
\end{lemma}
\begin{proof}
It suffices to assume that $E$ is compact and find a $y'\in\RR^m$ for which 
\begin{equation}\label{eq:612compact}
\int_{[0,1]^m} \HH^n( f(E_{y'}) \cap f(E_z) ) d\HH^m(z) \gtrsim_{n,m} L^{-n} \frac{\HH^{n+m}(E)^2}{\HH^n(f(E))}
\end{equation}
The result for more general sets then follows by applying this to a compact subset of $E$ of at least half the $\HH^{n+m}$-measure. Thus, we now assume $E$ is compact and prove \eqref{eq:612compact}.

We can also assume that $\HH^n(f(E))>0$. If not, then $\HH^{n+m}(h(E)) \leq \HH^{n+m}(f(E)\times [0,1]^m)=0$, so $\HH^{n+m}(E)=0$ as $h|_{E}$ is bi-Lipschitz, contradicting our assumption.

We now consider the function $K:X\times [0,1]^m \times [0,1]^m \rightarrow \{0,1\} \subseteq \RR$ defined by
$$ K(p,y,z) = \chi_{f(E_y)}(p)\chi_{f(E_z)}(p).$$

Recall the function $h(x,y)=(f(x,y),y)$. For fixed $t\in (0,1)$, the set
\begin{align*}
K^{-1}((t,\infty)) &= \{(p,y,z) \in X\times [0,1]^m\times [0,1]^m: p \in f(E_y) \text{ and } p\in f(E_z) \}\\
&=\{(p,y,z) \in X\times [0,1]^m\times [0,1]^m:(p,y)\in h(E) \text{ and } (p,z)\in h(E) \},
\end{align*}
is evidently compact and hence Borel. For $t>1$, this set is empty, and for $t<0$ this set is all of $X\times [0,1]^m\times [0,1]^m$. Thus, $K$ is Borel measurable on $X\times [0,1]^m\times [0,1]^m$.

We now consider the integral
$$ I = \int_X \int_{[0,1]^m} \int_{[0,1]^m} K(p,y,z)\, d\HH^m(y) d\HH^m(z) d\HH^n(p).$$

One application of Fubini's theorem allows us to rewrite this as
$$ I = \int_{[0,1]^m}\int_{[0,1]^m} \HH^n(f(E_y) \cap (f(E_z))\, d\HH^m(y) d\HH^m(z), $$
with an integrand that is Borel measurable for a.e. joint choice of $y$ and $z$ in $[0,1]^m \times [0,1]^m$.

Now write $\pi_X$ and $\pi_{\RR^m}$ for the projections from $X\times [0,1]^m$ to $X$ and $[0,1]^m$, respectively. Below, we follow the lead of \cite{AS12} and label each Hausdorff measure by a subscript denoting the space on which it is supported. For example, $\HH^n_{\RR^n}$ denotes Hausdorff $n$-measure on $\RR^n$ and $\HH^n_X$ denotes Hausdorff $n$-measure on $X$.

Using Fubini's theorem repeatedly, we can write
\begin{align}
I &= \int_X \int_{[0,1]^m} \int_{[0,1]^m} K(p,y,z) d\HH^m_{\RR^m}(y) d\HH^m_{\RR^m}(z) d\HH^n_X(p)\\
&= \int_X  \left(\int_{[0,1]^m} \chi_{f(E_y)}(p) \,d\HH^m_{\RR^m}(y) \right) \left(\int_{[0,1]^m} \chi_{f(E_z)}(p) \,d\HH^m_{\RR^m}(z) \right) d\HH^n_X(p)\\
&= \int_X  \HH_{X\times [0,1]^m}^m\left( \pi_{X}^{-1}(p) \cap h(E) \right)^2 d\HH^n_X(p)\\
&\geq \frac{1}{\HH^n_X(f(E))} \left(\int_X  \HH^m_{X\times [0,1]^m}\left( \pi_{X}^{-1}(p) \cap h(E) \right) d\HH^n_X(p)\right)^2 \label{eq:CS}\\
&= \frac{1}{\HH^n_X(f(E))}  \left(\HH^n_X \times \HH^m_{\RR^m} \right)(h(E))^2\\
&= \frac{1}{\HH^n_X(f(E))} \left( \int_{[0,1]^m} \HH^n_{X\times [0,1]^m}( \pi_{\RR^m}^{-1}(y) \cap h(E)) \,d\HH^m_{\RR^m}(y) \right)^2\\
&= \frac{1}{\HH^n_X(f(E))} \left( \int_{[0,1]^m} \HH^n_{X\times [0,1]^m}(h(E_y)) \,d\HH^m_{\RR^m}(y) \right)^2\\
&\geq \frac{1}{\HH^n_X(f(E))} (L^{-n})\left( \int_{[0,1]^m} \HH^n_{\RR^n}(E_y) \,d\HH^m_{\RR^m}(y) \right)^2 \label{eq:hbilip}\\
&= L^{-n} \frac{\HH^{n+m}_{\RR^{n+m}}(E)^2}{\HH^n_X(f(E))}
\end{align}
Most manipulations above are by Fubini's theorem. Equation \eqref{eq:CS} is by the Cauchy-Schwarz inequality, and \eqref{eq:hbilip} is a consequence of the fact that $h$ is $L$-bi-Lipschitz on $E$.

We thus have
$$  \int_{[0,1]^m}\int_{[0,1]^m} \HH^n(f(E_y) \cap (f(E_z))\, d\HH^m(y) d\HH^m(z) = I \geq L^{-n} \frac{\HH^{n+m}(E)^2}{\HH^n(f(E))}.$$
It follows that there is a choice of $y'\in [0,1]^m$ for which
$$ \int_{[0,1]^m} \HH^n(f(E_{y'}) \cap (f(E_z))\, d\HH^m(z)  \gtrsim_{n,m} L^{-n} \frac{\HH^{n+m}(E)^2}{\HH^n(f(E))}$$
with Borel measurable integrand.
\end{proof}

We now prove Proposition \ref{prop:directional}. The proof uses many of the ideas from the proof of \cite[Theorem 6.1]{AS12}, combined with a new iteration argument.

\begin{proof}[Proof of Proposition \ref{prop:directional}]
Let $F\subseteq Q_0$, $f\colon F \rightarrow X$, and $h(x,y)=(f(x,y),y)$ satisfy the assumptions of Proposition \ref{prop:directional}. Fix $\alpha>0$.

Our goal is to decompose $F$ into Hard Sard sets for $f$, plus a ``garbage set'' of small $\HH^{n+m}$-measure.

Let $\alpha'=\alpha/2$. Assume $\HH^{n+m}(F)\geq \alpha'$. (Otherwise, we stop, since the conclusion of Proposition \ref{prop:directional} then holds trivially with $F=G$.)

Given $z\in [0,1]^m$, we write $F_z$ for $F\cap (\RR^n \times\{z\})$ as above.

By Lemma \ref{lem:612}, there is a $y'\in \RR^m$ such that
$$ \int \HH^n(f(F_{y'}) \cap f(F_z)) d\HH^m(z) \gtrsim \frac{\HH^{n+m}(F)^2}{\HH^n(f(\overline{F}))} \geq (\alpha')^2 \approx_{\alpha} 1.$$
The implied constant in the first inequality depends only on $n$, $m$, and the bi-Lipschitz constant $L$ of $h$. Therefore we have a constant $\eta= \eta(n,m,\alpha, L)>0$ such that
\begin{equation}\label{eq:612-2}
\int_{[0,1]^m} \HH^n(f(F_{y'}) \cap f(F_z)) d\HH^m(z) \geq \eta.
\end{equation}

Let 
$$E^1=F \cap f^{-1}(f(F_{y'})).$$ 
Note that, since $h(x,y)=(f(x,y),y)$ and $h$ is $L$-bi-Lipschitz on $F$, we see that $f$ is $L$-bi-Lipschitz on $F_{y'}$.

Define two functions:
$$ p^1: f(F_{y'}) \rightarrow \RR^n \text{ by } p^1(z) = \pi_{\RR^n} \circ (f|_{F_{y'}})^{-1}(z),$$
and
\begin{equation}\label{eq:gdef}
 g^1: E^1\rightarrow \RR^{n+m} \text{ by } g^1(x,y) = (p^1(f(x,y)), y).
\end{equation}
We know that $h$ is $L$-bi-Lipschitz on $F\supseteq E^1$, and so $p^1$ is $L$-bi-Lipschitz on $f(F_{y'})$. It follows that $g^1$ is $L$-Lipschitz and $(L_1)$-bi-Lipschitz on $E^1$, with $L_1=2L^2$.

We first make an observation about the size of $E^1$.
\begin{claim}\label{claim:E1size}
We have $\HH^{n+m}(E^1) \geq L^{-2n}\eta$.
\end{claim}
\begin{proof}[Proof of Claim \ref{claim:E1size}]
This is as in \cite[p.1114]{AS12}, with slightly different notation. Note that, for any $z\in \RR^m$, we have
$$ g^1(E^1_z) = p(f(E^1_z) \cap f(F_{y'})) \times \{z\}.$$
In addition, we observe that $g^1$ only distorts the $x$-coordinate of points in $E^1$. In other words,
$$\pi_{\RR^m}(g^1(x,z)) = y \text{ if and only if } z=y.$$

Therefore, for each $z\in \RR^m$, we have
\begin{align*}
 \HH^n(g^1(E^1)_z) &= \HH^n(g^1(E^1_z))\\
&= \HH^n( (p^1(f(E^1_z) \cap f(F_{y'}))) \times \{z\})\\
&\geq L^{-n} \HH^n(f(E^1_z) \cap f(F_{y'}))\\
&= L^{-n} \HH^n(f(F^1_z) \cap f(F_{y'}))
\end{align*}
The inequality in the third line is because $p^1$ is $L$-bi-Lipschitz on $f(F_{y'})$, and the final equality is because $f(E^1_z)\cap f(F_{y'})= f(F_z)\cap f(F_{y'})$ by definition of $E^1$.

Since $g^1$ is $L$-Lipschitz, we conclude that
\begin{align*}
\HH^{n+m}(E^1)&\geq L^{-n}\HH^{n+m}(g(E^1)) \\
&= L^{-n}\int \HH^n(g(E^1)_z)d\HH^m(z)\\
&\geq L^{-2n} \int  \HH^n(f(F_z) \cap f(F_{y'}))d\HH^m(z)\\
&\geq L^{-2n}\eta
\end{align*}
using \eqref{eq:612-2}.
\end{proof}

Let $\alpha'' = \frac{\alpha\eta}{10C_{n+m} L^{3n}}$. Here $C_{n+m}$ is the constant from Lemma \ref{lem:contentmeasure} in dimension $d=n+m$. Note that $\alpha''$ depends only on $n,m, \alpha$, and $L$.

We now apply Corollary \ref{cor:bilipextension} to $g^1$ on $A=E^1$ with parameter $\kappa = \alpha''$. We obtain numbers $M'=M'(n,m,\alpha'')$ and $L'=L'(n,m,\alpha'')$ such that the set $E^1\subseteq F\subseteq Q_0$ admits a decomposition
$$ E^1 = E^1_1 \sqcup \dots \sqcup E^1_{M'} \sqcup G^1 $$
where
\begin{equation}\label{eq:gbilip}
 g^1|_{E^1_i} \text{is bi-Lipschitz, and moreover extends to an } L'-\text{bi-Lipschitz map on all } \RR^{n+m}
\end{equation}
and
$$ \HH^{n+m}(G^1) \leq C_{n+m}\HH^{n+m}_\infty (G^1) < C_{n+m}\alpha''.$$

Observe that $g^1|_{E^1_i}$ satisfies \eqref{eq:gshear} immediately from the definition of $g^1$ in \eqref{eq:gdef}.

Similarly as in \cite{AS12}, we now have the following claim.
\begin{claim}\label{claim:HSset}
For each $i\in\{1,\dots, M'\}$, the pairs $(E^1_i, g^1|_{E^1_i})$ are Hard Sard pairs for $f$, and the map $g^1$ satisfies \eqref{eq:gshear}.
\end{claim}
\begin{proof}[Proof of Claim \ref{claim:HSset}]
That $g^1$ satisfies \eqref{eq:gshear} is immediate from the definition of $g^1$ in \eqref{eq:gdef}.

We verify that $(E^1_i, g^1|_{E^1_i})$ satisfy the three conditions of Definition \ref{def:HSpair}. Condition \eqref{HS2} is established in \eqref{eq:gbilip}.

Write $F = f \circ (g^1)^{-1}$ on $g^1(E^1_i)$. Consider any $(x,y)= g^1(a,y) \in g^1(E^1_i)$. Then, by definition of $g^1$, $x=p^1(f(a,y))$, from which it follows that
\begin{equation}\label{eq:y'}
 F(x,y) = f(a,y) = f(x,y')
\end{equation}
and moreover that $(x,y')\in F_{y'}\subseteq F$.

Thus, if $(x,y_1)$ and $(x,y_2)$ are in $g^1(E^1_i)$, then
$$ F(x,y_1) = f(x,y') = F(x,y_2).$$
Conversely, if $(x_1,y_1)$ and $(x_2,y_2)$ are in $g^1(E^1_i)$ and $F(x_1, y_1) = F(x_2,y_2)$, then
$$ f(x_1, y') = f(x_2, y') \text{ and } (x_1,y'), (x_2,y')\in F,$$
and this implies that $x_1=x_2$ because the map $h(x,y) = (f(x,y),y)$ is bi-Lipschitz on $F$. This verifies Condition \eqref{HS3} of Definition \ref{def:HSpair}.

To verify Condition \eqref{HS4} of Definition \ref{def:HSpair}, consider points $(x_1,y_1)$ and $(x_2,y_2)$ in $g^1(E^1_i)$. Then, using \eqref{eq:y'} and the fact that $h(x,y)$ is $L$-bi-Lipschitz on $F\supseteq E^1_i$, we have
\begin{align*}
d\left( (F(x_1, y_1),y_1) , (F(x_2,y_2),y_2) \right) &= d\left( (f(x_1,y'),y_1) , (f(x_2,y'),y_2) \right) \\
&\approx d(f(x_1,y'), f(x_2,y')) + |y_1 - y_2|\\
&\approx |x_1 - x_2| + |y_1-y_2|\\
&\approx |(x_1,y_1) - (x_2,y_2)|
\end{align*}
The implied constants here depend only on $n$, $m$, and $L$. This verifies Condition \eqref{HS4} of Definition \ref{def:HSpair}.

\end{proof}

To recap, we have now given the following decomposition of $F$:
$$ F = (F \setminus E^1) \sqcup E^1 = (F \setminus E^1) \sqcup E^1_1 \sqcup \dots \sqcup E^1_{M'} \sqcup G^1,$$
where
\begin{itemize}
\item each $(E^1_i, g^1|_{E^1_i})$ is a Hard Sard pair,
\item $\HH^{n+m}(F^1 \setminus E^1) \leq  \HH^{n+m}(F^1) - \HH^{n+m}(E^1) \leq \HH^{n+m}(F^1) - L^{-2n}\eta$, and
\item $\HH^{n+m}(G^1) < C_{n+m}\alpha''$ 
\end{itemize}

Writing $F^1=F$, let $F^2 = F^1 \setminus E^1$. If $\HH^{n+m}(F^2)\geq \alpha'$, we may repeat the exact same procedure on $F^2$. Continuing inductively in this way, if $F^{k-1}$ is defined and has a decomposition as above, set $F^{k} = F^{k-1} \setminus E^{k-1}$. If $\HH^{n+m}(F^k) \geq \alpha'$, we may follow the same procedure to obtain a decomposition of $F^k$:
$$ F^k = (F^k \setminus E^k) \sqcup E^k = (F^k \setminus E^k) \sqcup E^k_1 \sqcup \dots \sqcup E^k_{M'} \sqcup G^k,$$
where
\begin{enumerate}[(i)]
\item\label{iteration1} each $(E^k_i, g^k|_{E^1_i})$ is a Hard Sard pair satisfying \eqref{eq:gshear},
\item\label{iteration2} $\HH^{n+m}(F^k \setminus E^k) \leq \HH^{n+m}(F^{k}) - L^{-2n}\eta$, and
\item\label{iteration3} $\HH^{n+m}(G^k) <  C_{n+m}\alpha''$ 
\end{enumerate}
Notice that the constants $L, \eta, L_1, M'$ remain uniform throughout this process.

We end this iteration when $F^N = F^{N-1}\setminus E^{N-1}$ has $\HH^{n+m}(F^N)<\alpha'$. This procedure ends after at most $\frac{L^{2n}}{\eta}$ steps by property \eqref{iteration2}. (It may even be the case that $\HH^{n+m}(F^1)<\alpha'$ at the very first step, in which case $E^1=\emptyset$ and we end immediately.)

In the final tally, we have a collection of at most $M' \cdot \frac{L^{2n}}{\eta}$ Hard Sard pairs in direction $P$ contained in $F$.

The remainder of $F$ consists of all the sets $G^k$ ($k=1, \dots, N$) that we constructed along the way, and the final stage $F^N$ at which we stopped the iteration. We now show that those remaining sets have small total area.

We generated at most  $\frac{L^{2n}}{\eta}$ sets $G^k$ in the above iteration scheme. Each of these sets has $\HH^{n+m}(G_k) < C_{n+m} \alpha''$ by \eqref{iteration3}. 

The final stage $F^N$ at which we end the iteration has, by construction,
$$ \HH^{n+m}(F^N) < \alpha'.$$

Thus, the complement of all Hard Sard sets $E^k_i$ that we formed in this procedure has total $\HH^{n+m}$ measure at most
$$ \left(\frac{L^{2n}}{\eta}\right)(C_{n+m} \alpha'') + \alpha'.$$
By our choice of the constants $\alpha'$ and $\alpha''$, this quantity is less than $\alpha$. This completes the proof of Proposition \ref{prop:directional}.
\end{proof}

\section{Coding and splitting cubes}\label{sec:codesplit}

In this section and the next, we proceed with some preliminary work necessary to prove Proposition \ref{prop:hbilip}, which along with Proposition \ref{prop:directional} is the other main ingredient in the proof of Theorem \ref{thm:hardersarder}.

We need two lemmas, used in the proof of Proposition \ref{prop:hbilip} to code and split cubes into useful families.

The first is a ``coding'' lemma which is a slight variant of now-standard arguments used, e.g., in \cite[p. 199-121]{Jo88}, \cite[p. 81-82]{Davidbook}, \cite[p. 8]{Sc09}. We give a brief proof here.

\begin{lemma}\label{lem:coding}
Let $Q_0$ be the unit cube of $\RR^d$, and let $f:Q_0 \rightarrow X$ be a $1$-Lipschitz map into a metric space. Fix $C_0\geq 1$, $\epsilon>0$ and $\eta>0$. Then we can decompose $Q_0$ into sets
$$ Q_0 = A_1 \cup A_2 \cup \dots A_{M_{\md}} \cup G_{\md} $$
with the following two properties:
\begin{enumerate}[(i)]
\item If $x,y\in A_i$ for some $1\leq i \leq M_{\md}$ and $Q\in \Delta$ is a dyadic cube of minimal side length such that $x,y\in 3Q$, then
$$ \md(C_0 Q) <\epsilon.$$
\item $\HH^{d}_\infty(G_{\md})<\eta$. 
\end{enumerate}
The number of sets $M_{\md}$ depends only on $\epsilon$, $\eta$, and $d$. 
\end{lemma}
\begin{proof}
We may assume that $f$ extends to a $1$-Lipschitz map defined on all of $\RR^d$. (This can be done by simply extending $f$ ``radially'' to agree with its values on the boundary of $Q_0$, or by applying the Kuratowski embedding theorem to embed $X$ in $\ell_\infty(X)$ and performing a $1$-Lipschitz extension into $\ell_\infty(X)$.)

We will use  Theorem \ref{thm:quantdiff}, from which we deduce (with appropriate constants $C'$ and $\epsilon'$, to be shortly fixed, depending on the dimension $d$) 
that for $N$ large enough we have
$$G:=\{x\in Q_0: x\in R_1\subsetneq R_2\subsetneq...\subsetneq R_N\subset Q_0; \md(C'R_i)\geq \epsilon'\}$$
has measure, hence content, less than $\eta/2^d$ (here, we assume $R_i\in\Delta$).

Consider now $x\in Q_0\setminus G$.  Each such $x$ has at most $N-1$ cubes $R\ni x$ such that $\md(C'R)>\epsilon'$, which we may denote by
$R_1(x),...R_{N(x)}$, where $N(x)<N$.
By using an alphabet of  $m(d)$ letters, we may assign
for each $x\in Q_0$, a word  $w(x)$, 
a sequence of $<N$ letters,
such that if $w(x)=w(y)$ then a minimal cube $Q$ such that  $Q\ni x,y$ has $\md(C'Q)<\epsilon'$.

We now use the ``$\frac13$ trick'': apply the above construction to all $2^d$ dyadic grids formed by shifting the standard dyadic partition by $\frac{1}{3}$ in any combination of coordinate directions. (See \cite[p. 5]{Sc09}.)

Note that if $Q$ is a smallest cube (from the original dyadic grid) such that $3Q\ni x,y$, then $C_0Q$ is contained in $C'Q'$, for some minimal cube $Q'$ containing $x,y$ in one of the shifted dyadic grids and satisfying $\side(Q')\approx_d \side(Q)$.

Thus, after repeating this $2^d$ times on all the new dyadic grids, each $x$ has words 
$$w_1(x),\dots,w_{C(d)}(x)$$
such that if for all $j$
$w_j(x)=w_j(y)$, then a smallest cube $Q$ such that $3Q\ni x,y$ has   $\md(C_0Q)<\epsilon$.
Here we determine $C'$ and $\epsilon'$ to depend linearly on their counterparts $C_0$ and $\epsilon$, with a dependance depending on the ambient dimension $d$.
\end{proof}

The next lemma allows for splitting a family of almost-disjoint cubes into a controlled family of well-separated cubes, at the cost of throwing a way a set of small measure.

\begin{lemma}\label{lem:splitting}
Let $\cQ$ be a collection of almost-disjoint cubes in $Q_0\subseteq \RR^d$, $\eta>0$, and $\Lambda$ an odd natural number. Then we can partition $\cQ$ into families $\{\cQ_k\}_{k=1}^{k_0}$ and $\cB$ such that
\begin{enumerate}[(i)]
\item every cube of $\cQ$ is in exactly one of the sets $\cQ_k$ or $\cB$,
\item if $Q, Q'\in \cQ_k$, then $\Lambda Q \cap \Lambda Q'=\emptyset$, and 
\item the $\HH^{n+m}$-measure of the union of all cubes in $\cB$ is $<\eta$.
\end{enumerate}
The number of families $k_0$ depends only on $\eta$, $\Lambda$, and $d$.
\end{lemma}
\begin{proof}
It suffices to assume that the family $\cQ$ is finite, which we do for convenience.

Let $\cQ_1 \subseteq \cQ$ be a subset defined inductively as follows: Put the largest cube in $\cQ$ in $\cQ_1$, breaking ties arbitrarily. Then at each step add to $\cQ_1$ the largest cube $Q$ in $\cQ\setminus \cQ_1$ such that $\Lambda Q\cap \Lambda Q'=\emptyset$ for all $Q'\in \cQ_1$. Continue this process until no more cubes can be added to $\cQ_1$.

Denoting Lebesgue measure by $|\cdot|$, we will first prove the following:
\begin{claim}\label{claim:separation}
There is a constant $c=c(d,\Lambda )>0$ such that
$$ |\cup_{Q\in \cQ_1} Q| \geq c |\cup_{Q\in \cQ} Q|.$$
\end{claim}
\begin{proof}
Given $Q\in \cQ_1$, let
$$\cE_Q = \{Q'\in\cQ : \Lambda Q\cap \Lambda Q' \neq \emptyset \text{ and } \side(Q')\leq\side(Q)\}.$$
For each $Q\in\cQ_1$ and $Q'\in \cE_Q$, we have $Q' \subseteq (2\Lambda +1)Q$. Therefore,
$$ |\cup_{Q'\in \cE_Q} Q' | \leq |(2\Lambda +1)Q| = (2\Lambda +1)^d|Q|.$$

Now, if $Q'\in \cQ\setminus \cQ_1$, then $Q'\in \cE_Q$ for some $Q\in \cQ_1$. (Otherwise, $Q'$ would have been placed in $\cQ_1$.)

Therefore,
$$ |\cup_{Q'\in \cQ \setminus \cQ_1} Q'| \leq \sum_{Q\in \cQ_1} | \cup_{Q'\in \cE_Q} Q'| \leq (2\Lambda +1)^d  \sum_{Q\in \cQ_1} |Q|.$$
It follows that
$$|\cup_{Q\in \cQ} Q| =  |\cup_{Q\in \cQ_1} Q| + |\cup_{Q'\in \cQ \setminus \cQ_1} Q'| \leq (1+(2\Lambda +1)^d)|\cup_{Q\in \cQ_1} Q|, $$
which proves the claim.
\end{proof}

Once the claim is proven, the lemma follows: For each $i\geq 2$, apply the same construction to $\cQ \setminus \cup_{j=1}^{i-1}\cQ_j$ to obtain collections 
$$\cQ_i \subseteq \cQ \setminus \cup_{j=1}^{i-1}\cQ_j \subseteq \cQ$$
satisfying (ii) and having
$$ |\cup_{Q\in \cQ_i} Q| \geq c |\cup_{Q\in \cQ\setminus \cup_{j=1}^{i-1}\cQ_j } Q|.$$

It follows that, for each $i\geq 1$,
\begin{align*}
|\cup_{Q\in \cQ\setminus \cup_{j=1}^{i}\cQ_j } Q| &=  |\cup_{Q\in \cQ\setminus \cup_{j=1}^{i-1}\cQ_j } Q| - |\cup_{Q\in \cQ_i} Q|\\
&\leq (1-c) |\cup_{Q\in \cQ\setminus \cup_{j=1}^{i-1}\cQ_j } Q|
\end{align*}
Hence
$$ |\cup_{Q\in \cQ\setminus \cup_{j=1}^{i}\cQ_j } Q| \leq (1-c)^i|\cup_{Q\in \cQ} |Q| \leq (1-c)^i \text{ for each } i\geq 1$$
Thus, for $k_0$ sufficiently large depending on $\eta$ and $c=c(d,\Lambda)$ (and recalling the comparability of Lebesgue and Hausdorff measures), we can ensure the set
$$\cB =  \cQ\setminus \cup_{j=1}^{k_0}\cQ_j $$
has a union with total $\HH^{n+m}$-measure $<\eta$. This completes the proof of the lemma.

\end{proof}

\section{Cubes compressed in many directions}\label{sec:compressed}

In the proof of Theorem \ref{thm:hardersarder} (via Proposition \ref{prop:hbilip}), we will need to discard a collection of cubes which, while they may have small $\md$, are ``compressed'' by $f$ in many different directions. This collection will eventually form part of the ``garbage set'' $G$ in Proposition \ref{prop:hbilip} and Theorem \ref{thm:hardersarder}.

For the remainder of this section, we work under the following assumptions: $Q_0$ is the unit cube of $\RR^{n+m}$ and $f\colon Q_0 \rightarrow X$ is a $1$-Lipschitz map from $Q_0$ into a metric space. We will write $\md$ for $\md_f$ in this section and the following one.

Let $\PP$ be the set of all \textit{coordinate $n$-planes} in $\Gr(n,n+m)$. That is,
\begin{equation}\label{eq:PPdef}
 \PP = \{ \text{span}(\{e_{i_1}, e_{i_2}, \dots, e_{i_n}\}) : 1\leq i_1 < \dots < i_n \leq n+m\},
\end{equation}
where $e_i$ represent the $n+m$ standard basis vectors of $\RR^{n+m}$.

Recall the constant $C_0$ defined in \eqref{eq:C0def}. For each cube $Q\in \Delta$, we fix a seminorm $\|\cdot\|_Q$ on $\RR^{n+m}$ that minimizes $\md(C_0 Q)$.

Our goal in this section is to prove the following lemma.

\begin{lemma}\label{lem:compressedcontent}
Fix $\epsilon>0$ and $\delta>0$. Let $\QQ=\QQ(\epsilon,\delta)$ be the collection of all cubes such that
\begin{enumerate}[(i)]
\item $ \md(C_0 Q) < \epsilon$, and
\item for all $P\in \PP$, there is a unit vector $v_P\in P$ such that $\|v_P\|_Q < \delta$.
\end{enumerate}
Then
$$ \HH^{n,m}_\infty(f, \cup_{\QQ} (3Q)) \lesssim \epsilon+ \delta. $$ 
The implied constant depends only on $n$ and $m$.
\end{lemma}

As a first step, we need the following linear algebra fact.

\begin{lemma}\label{lem:coordplane1}
Let $K$ be a subspace of dimension $\leq m$ in $\RR^{n+m}$. Then $K \cap P= \{0\}$ for at least one $P\in\PP$.

In other words, $K$ cannot contain a non-zero vector in every coordinate $n$-plane.
\end{lemma}
\begin{proof}
It suffices to show that an $m$-plane $K$ in $\RR^{n+m}$ cannot contain a non-zero vector in every coordinate $n$-plane. The proof is by induction on $n$, for each fixed $m\geq 1$. 

First, suppose that $n=1$, and that $K\in \Gr(m,1+m)$ contains a non-zero vector in every coordinate $1$-plane. Then $K$ contains a non-zero multiple of every standard basis vector $e_i$, which implies that $K \supseteq \vspan(\{e_i:i=1\dots 1+m\}) = \RR^{1+m}$ and yields a contradiction.

Now suppose that $n\geq 1$ and $K\in \Gr(m,n+m)$. As above, there must be a standard basis vector $e_{i_0}$ such that $K \cap \vspan(\{e_{i_0}\}) = \{0\}$; if not, $K$ would be all of $\RR^{n+m}$. Let
$$ V = \vspan\{ e_1, e_2, \dots, e_{i_0 - 1}, e_{i_0 + 1}, \dots, e_{n+m}\} .$$
Let $K' = \pi_V(K)$. Since $K \cap \ker(\pi_V) = \{0\}$, the space $K'$ is an $m$-dimensional subspace of $V\cong\RR^{(n-1)+m}$.

By induction, $K'$ cannot contain a vector in each coordinate $(n-1)$-plane of $V$. In other words, there is a collection
$$ \{e_{i_1}, \dots, e_{i_{n-1}}\}$$
of standard basis vectors (none of which are $e_{i_0}$) such that 
$$ K' \cap \vspan( \{e_{i_1}, \dots, e_{i_{n-1}}\} ) = \{0\}.$$

Let
$$ P = \vspan(\{e_{i_0}\} \cup \{e_{i_1}, \dots, e_{i_{n-1}}\} ) \in \PP.$$
This is a coordinate $n$-plane in $\RR^{n+m}$. If $v$ is a non-zero vector in $K\cap P$, then
$$ \pi_V(v) \in K' \cap  \vspan( \{e_{i_1}, \dots, e_{i_{n-1}}\} ) = \{0\},$$
and so $v \in \vspan(\{e_{i_0}\})$. But $e_{i_0}$ was chosen above so that $K \cap \vspan(\{e_{i_0}\}) = \{0\}$, so $v$ must be zero. 

Thus, $K \cap P = \{0\}$, as desired.

\end{proof}

A compactness argument then yields the following quantitative version of the previous lemma.
\begin{lemma}\label{lem:coordplane}
For each $n,m\geq 1$, there is a constant $c=c(n,m)$ with the following property:

Let $K$ be a subspace of dimension $\leq m$ in $\RR^{n+m}$. Then there is a coordinate $n$-plane $P\in\PP$ such that 
$$ \dist(w, K) \geq c > 0$$
for all unit vectors $w\in P$.
\end{lemma}
\begin{proof}
First of all, it clearly suffices to prove the lemma assuming that $\dim(K)=m$.

Suppose that the lemma were false for some fixed $n,m\geq 1$.

Then there would be a sequence $K_j\in \Gr(m,n+m)$ such that, for every $P\in \PP$, there is a unit vector $w^P_j \in P$ with
$$ \dist(w^P_j, K_j) < \frac{1}{j}.$$

The sequence $K_j$ has a subsequence converging in the metric $D$ defined in subsection \ref{subsec:grassmannian} to a subspace $K\in \Gr(m,n+m)$. In addition, we may pass to further subsequences for which each sequence $\{w^P_j\}_{j=1}^\infty$ converges to a unit vector $w^P\in P$. It then follows that
$$ w^P \in K\cap P.$$
Thus, $K$ is an $m$-plane containing a non-zero (indeed, unit) vector in every $P\in\PP$. This contradicts Lemma \ref{lem:coordplane1}. 
\end{proof}

We now prove Lemma \ref{lem:compressedcontent}.

\begin{proof}[Proof of Lemma \ref{lem:compressedcontent}]
Recall the definition of $C_0$ from \eqref{eq:C0def}, and let $c = c(n,m)$ denote the constant from Lemma \ref{lem:coordplane}.

We may assume in proving the lemma that $\epsilon$ and $\delta$ are both small, depending on $n$ and $m$, e.g., that
\begin{equation}\label{eq:wlogsmall}
 \delta + \epsilon < c(100 C_0)^{-1},
\end{equation}
otherwise the lemma is trivial.

Recall that if a cube $Q$ is in the collection $\QQ$ defined in the statement of Lemma \ref{lem:compressedcontent}, then $$\md(C_0Q)<\epsilon$$ and for every $P\in \PP$, there is unit vector $v_P\in P$ such that 
\begin{equation}\label{eq:vp}
\|v_P\|_Q < \delta.
\end{equation}

We establish Lemma \ref{lem:compressedcontent} via some intermediary claims.

\begin{claim}\label{claim:kernel}
If $Q\in \QQ$, there is an $(m+1)$-plane $K_0\in \Gr(m+1,n+m)$ such that 
\begin{equation}\label{eq:kernel}
\|v\|_{Q} \leq C_1\delta|v| \text{ for all } v\in K_0,
\end{equation}
where $C_1$ is a constant depending only on $n$ and $m$.
\end{claim}
\begin{proof}
For each $P\in\PP$, there is a unit vector $v_P\in P$ satisfying  \eqref{eq:vp}. Fix $c'=\frac{c}{2}$. Let $S=\{v_1, \dots, v_\ell\}$ be a maximal subset of $\{v_P:P\in \PP\}$ that satisfies
\begin{equation}\label{eq:Sdef}
\dist(v_i, \text{span}(\{v_1, \dots, v_{i-1}\}) \geq c' \text{ for each } i\in\{1,\dots, \ell\}.
\end{equation}
In other words, $S$ is a maximal ``quantitatively linearly independent'' subset of $\{v_P:P\in\mathcal{P}\}$ (with parameter $c'$).

We will show that 
$$ K = \text{span}(S) $$
has $\dim(K) \geq m+1$ and satisfies \eqref{eq:kernel}.

First, we argue that $K$ satisfies \eqref{eq:kernel}. Any $v\in K$ can be written as
$$ v = \sum_{i=1}^\ell a_i v_i$$
where
$$ \sum_{i=1}^\ell |a_i| \leq C_1|v|,$$
{because $S$ satisfies \eqref{eq:Sdef}.} 
Here $C_1$ is a constant depending on $n$, $m$, and $c'$, and thus ultimately only on $n$ and $m$.

Hence
$$ \|v\|_{Q} \leq \sum_{i=1}^\ell |a_i| \|v_i\|_{Q}  \leq \delta \sum_{i=1}^\ell |a_i|   \leq C_1\delta |v|,$$
where $C_1$ depends only on $n$ and $m$. This proves that $K$ satisfies \eqref{eq:kernel}.

Now we show that $\dim(K)\geq m+1$. Suppose, towards a contradiction, that $\dim(K) \leq m$. Then by Lemma \ref{lem:coordplane}, there is a coordinate $n$-plane $P\in\PP$ such that
\begin{equation}\label{eq:wfar}
 \dist(w, K) \geq c > 0
\end{equation}
for all unit vectors $w\in P$.

Let $v_P$ be the unit vector associated to $P$, as in \eqref{eq:vp}. In particular, since $v_P\in P$, we have
$$ \dist(v_P, K) \geq c > 0.$$

In that case, however, we should have appended $v_P$ to $S$. In other words, this contradicts the maximality of the set $S$ defined above.

Therefore, $K$ must have dimension at least $m+1$, and it satisfies \eqref{eq:kernel}. It follows that $K$ contains an $(m+1)$-dimensional subspace $K_0$ that satisfies \eqref{eq:kernel}, and hence proves Claim \ref{claim:kernel}.
 
\end{proof}

\begin{claim}\label{claim:smallimage}
Let $Q\in \QQ$ and let $R$ be either $Q$ or a neighbor of $Q$ of the same scale. Then 
$$ \HH^{n}_\infty(f(R)) \lesssim (\delta+\epsilon)\side(R)^n,$$
where the implied constant depends only on $n$ and $m$.
\end{claim}
\begin{proof}[Proof of Claim \ref{claim:smallimage}]
Assume by rescaling that $\side(Q)=\side(R)=1$. Let $K_0$ be an $(m+1)$-plane as in Claim \ref{claim:kernel}, so that
$$\|v\|_{Q} \leq C_1 \delta |v| \text{ for all } v\in K_0.$$
Let $P = K_0^{\bot}$, which is an $(n-1)$-plane in $\RR^{n+m}$.

Let $P_0$ be an affine $(n-1)$-plane parallel to $P$ and passing through the center of $R$. Fix $x\in R$ and let $x_{P_0}$ denote the closest point to $x$ in $P_0$. Note that $x$ and $x_{P_0}$ are in $C_0Q$ by our choice of $C_0$ in \eqref{eq:C0def}. 

We have
$$ d(f(x) , f(x_{P_0})) \leq \|x-x_{P_0}\|_{Q} + C_0\epsilon \lesssim C_1C_0\delta + C_0\epsilon \leq C_2(\delta+\epsilon),$$
since $x-x_{P_0}\in K_0$ and $|x-x_{P_0}|\leq \diam(C_0Q)\lesssim C_0$. Here $C_2$ is again a constant depending only on $n$ and $m$.

Since $x\in R$ was arbitrary, we have shown that
$$ f(R) \subseteq N_{C_2(\delta+\epsilon)}f(P_0 \cap C_0Q).$$ 
Given $r>0$, we can cover $P_0 \cap C_0Q$ by $\lesssim r^{-(n-1)}$ balls of radius $r$, with implied constant depending only on $n$ and $m$. Hence, we can cover $f(R)$ by $\lesssim r^{-(n-1)}$ balls of radius $r+C_2(\delta+\epsilon)$. Choosing $r=C_2(\delta+\epsilon)$ allows us to bound
$$ \HH^n_\infty(f(R)) \lesssim (C_2(\delta+\epsilon))^{-(n-1)}(2C_2(\delta+\epsilon))^n \lesssim \delta+\epsilon.$$
This proves Claim \ref{claim:smallimage}.
\end{proof}

We now use Claim \ref{claim:smallimage} to complete the proof of Lemma \ref{lem:compressedcontent}. Let $\hat{\QQ}$ denote the collection of all cubes that are either in $\QQ$, or adjacent to an element of $\QQ$ of the same scale. Let $\{R_j\}$ enumerate all maximal cubes of $\hat{Q}$. Note that $\cup_{Q\in \QQ} (3Q) \subseteq \cup_{j} R_j$. Thus, using Claim \ref{claim:smallimage},
\begin{align*}
\HH^{n,m}_\infty(f,\cup_{\QQ} (3Q) ) &\leq \sum_j \HH^n_\infty(f(R_j))\side(R_j)^m\\
&\lesssim (\delta+\epsilon)\sum \side(R_j)^{n+m}\\
&\leq \delta + \epsilon.
\end{align*}
This completes the proof of Lemma \ref{lem:compressedcontent}.
\end{proof}

\section{Adding projections to form bi-Lipschitz maps}\label{sec:hbilip}

Our goal in this section is to prove Proposition \ref{prop:hbilip}. Thus, we will decompose the unit cube into pieces on which, after a bi-Lipschitz change of coordinates, our given Lipschitz map can be supplemented by the projection to the $x$-axis to form a bi-Lipschitz map.

For the entirety of Section \ref{sec:hbilip}, let $Q_0$ be the unit cube of $\RR^{n+m}$ and $f\colon Q_0 \rightarrow X$ be a $1$-Lipschitz mapping into a metric space. 
\begin{remark}
No condition on $\HH^n(f(Q_0))$ is assumed for this particular section.
\end{remark}

These are the same standing assumptions as in Section \ref{sec:compressed}.

Recall the definition of the collection of coordinate $n$-planes $\PP$ in \eqref{eq:PPdef}. Let $P_0\in\PP$ denote the span of the first $n$ standard basis vectors in $\RR^{n+m}$. Thus, if we write $\RR^{n+m}=\RR^n \times \RR^m$ in the usual way, then $\pi_{P_0}(x,y)=x$ and $\pi_{P_0^{\bot}}(x,y)=y$.

Note that the maps 
$$ (x,y) \mapsto (f(\phi_i^{-1}(x,y)), y),$$
described in Proposition \ref{prop:hbilip}, which we will discuss below, can be written using this notation as
$$(f\circ \phi_i^{-1}, \pi_{P_0^{\bot}}).$$

As one final remark before beginning the proof, part of the conclusion of Proposition \ref{prop:hbilip} is that the sets $F_i$ and $G$ we construct will be Borel. However, if $F_i$ is one of the sets in the proposition, then $\overline{F_i}$ has the same property, so this will not be a concern in the remainder of the proof.

We now begin the proof of Proposition \ref{prop:hbilip} in earnest, which will take a number of steps. Fix a parameter $\alpha>0$ as in the statement of the proposition.

A number of constants, such as $C_0$ defined in \eqref{eq:C0def}, depend only on $n$ and $m$, and will often be suppressed using the $\lesssim_{n,m}$ notation. 

At the moment we fix 
\begin{equation}\label{eq:alpha'def}
\alpha'=\alpha/(10C),
\end{equation}
where $C\geq 1$ is larger than the implied constant in Lemma \ref{lem:areaboundscontent}, which depends only on $n+m$. Thus, $\alpha'$ depends only on $\alpha,n,m$.

The most important further positive constants we will define will be $\delta$, chosen sufficiently small depending only on $\alpha$, $n$, and $m$; $\delta'$ chosen sufficiently small depending only on $\alpha,n, m,$ and $\delta$; and $\epsilon$, chosen sufficiently small depending on $\alpha,n,m,\delta,$ and $\delta'$. The needed requirements will be specified in the course of the proof. A number of other constants will be defined based on these as we go.

Recall that for each cube $Q\in \Delta$, we have fixed a seminorm $\|\cdot\|_Q$ that minimizes the quantity $\md(C_0Q)$.

We now define three sub-collections of cubes in $\Delta$ that we will use in the remainder of the proof.

\begin{definition}\label{def:cubefamilies}
Given positive constants $\delta$ and $\epsilon$ as above, define:
\begin{itemize}
\item $\QQ_{\md} := \{Q\in \Delta : \md(C_0Q)\geq \epsilon\}$.
\item $\QQ_{\text{compressed}}:=\{Q\in \Delta: \md(C_0 Q) < \epsilon \text{ and each plane } P\in \PP \text{ contains a unit vector } v_P\in P\text{ with } \|v_P\|_Q < \delta\}$.
\item $\QQ_{\text{good}}:=\Delta \setminus (\QQ_{\md} \cup \QQ_{\text{compressed}})$.
\end{itemize}
\end{definition}

Thus, the fact that a given cube $Q \in \QQ_{\text{good}}$ means that
$$ \md(C_0Q) < \epsilon$$
and that there is an $n$-plane $P_{Q} \in \PP$ such that
\begin{equation}\label{eq:PQdef}
\|v\|_{Q} \geq \delta \text{ for all } v\in P_{Q}.
\end{equation}

We assign to each cube $Q\in\QQ_{\text{good}}$ a fixed coordinate $n$-plane $P_Q\in \PP$ with the above property.

\subsection{Initial decomposition of $Q_0$ into starting cubes}

We begin with the following initial decomposition of $Q_0$.

\begin{lemma}\label{lem:initial}
There is a constant $K_1 = K_1(\alpha,\delta,\epsilon,n,m)$ such that
$$\HH^{n,m}_\infty\left(f, \bigcup_{Q\in\Delta_{k} \setminus \QQ_{\text{good}}} Q\right) < \alpha'$$
for some $k\leq K_1$.
\end{lemma}
\begin{proof}
Fix $K_1\in\mathbb{N}$ arbitrary for the moment. 

Then, using Theorem \ref{thm:quantdiff}, 
$$ \sum_{k=1}^{K_1} \sum_{Q\in \Delta_k \cap \QQ_{\md}} |Q| \lesssim_{n,m,\epsilon} 1.$$

Therefore, for some choice of $1\leq k \leq K_1$, 
$$ \HH^{n,m}_\infty(f, \bigcup_{Q\in \Delta_{k}  \cap \QQ_{\md}} Q) \leq  \sum_{Q\in \Delta_k \cap \QQ_{\md}} |Q| \lesssim_{n,m,\epsilon} \frac{1}{K_1}.$$

Furthermore, Lemma \ref{lem:compressedcontent} implies that
$$ \HH^{n,m}_\infty\left(f, \bigcup_{Q\in\Delta_{k}  \cap \QQ_{\text{compressed}}} Q\right) \lesssim_{n,m} \delta+\epsilon.$$
It follows that
$$\HH^{n,m}_\infty\left(f, \bigcup_{Q\in\Delta_{k} \setminus \QQ_{\text{good}}} Q\right) < \alpha'$$
if $K_1$ is chosen large depending on $n,m,\alpha, \alpha'$, and $\epsilon$, and $\delta$ and $\epsilon$ are small depending on $n,m,\alpha$. (This uses Lemma \ref{lem:subadditive}.)

\end{proof}

\subsection{Stopping time argument}\label{sec:stoppingtime}
Let $k$ be as in Lemma \ref{lem:initial}. Let $Q^0$ be a cube in $\Delta_k \cap \QQ_{\text{good}}$, which will be fixed for the next few subsections. We will view $Q^0$ as the top cube of a certain stopping time argument.

The fact that $Q^0\in \QQ_{\text{good}}$ means that
$$ \md(C_0Q^0) < \epsilon$$
and that there is an $n$-plane $P_{Q^0} \in \PP$ such that
$$ \|v\|_{Q^0} \geq \delta \text{ for all unit } v\in P_{Q^0}.$$

Let $S^0 = \{Q^0\}$. We inductively define collections $S^i \subseteq \QQ_{\text{good}}$, each of which consists of pairwise disjoint cubes, as follows. Assume we have defined
$$ S^{i-1} = \{ Q^{i-1}_1, \dots, Q^{i-1}_{k_{i-1}}\}.$$
A cube $Q$ will be placed in $S^i$ if it satisfies the following conditions:
\begin{enumerate}
\item\label{switch1} $Q\in \QQ_{\text{good}}.$
\item\label{switch2} $Q\subseteq Q^{i-1}_j$ for some cube $Q^{i-1}_j \in S^{i-1}$.
\item\label{switch3} there exists a unit vector $v\in P_{Q^{i-1}_j}$ with 
$$\|v\|_{Q} < \delta',$$
\item\label{switch4} $Q$ is a maximal sub-cube of $Q^{i-1}_j$ satisfying \eqref{switch1} and \eqref{switch3}.
\end{enumerate}

The above conditions define a disjoint family of cubes $S^i$ with the property that each cube of $S^i$ is in $\QQ_{\text{good}}$ and is contained in a cube of $S^{i-1}$.

One can view the above construction of the families $S^i$ in the following way: We begin with a starting cube $Q^0\in\QQ_{\text{good}}$ with an associated ``good plane'' $P_{Q^0}$, in the sense that \eqref{eq:PQdef} holds for $Q^0$ and $P_{Q^0}$. We proceed down each branch of the tree of descendants of $Q^0$. We stop at the first time we see a cube $Q^1\in \QQ_{\text{good}}$ for which $P_{Q^0}$ is a ``bad plane'' for $Q^1$, in the sense that there is a unit vector $v\in P_{Q^0}$ such that
$$ \|v\|_{Q^1} < \delta'.$$
In that case, we add $Q^1$ to $S^1$, find a new good plane $P_{Q^1}$ for $Q^1$, and continue the process. The largest good descendant of $Q^1$ for which $P_{Q^1}$ is no longer a good plane will be added to $S^2$, etc. 

Thus, cubes in $S:=\cup_{i=1}^\infty S^i$ are those good cubes for which the appropriate plane ``switches'' to another element of $\PP$, in a quantitative way.

\subsection{Packing condition}

Fix a $k$ as in Lemma \ref{lem:initial} and $Q^0\in \Delta_k \cap \QQ_{\text{good}}$.

Our next goal is to prove that most points of $Q^0$ cannot lie in too many nested cubes of $S$, i.e., that $S^i$ for large $i$ has small $\HH^{n+m}$-measure.

\begin{lemma}\label{lem:packed}
There is a choice of $K=K(n,m,\alpha',\delta,\delta',\epsilon) \in \mathbb{N}$ such that
$$ \HH^{n+m}\left(\cup_{Q\in S^K} Q \right) < \alpha'\HH^{n+m}(Q^0).$$
\end{lemma}

The main step in the proof of Lemma \ref{lem:packed} is the following.

\begin{lemma}\label{lem:packedfubini}
Let $Q\in S^i$ for some $i\geq 1$. Then the collection of all $R\in S^{i+1}$ that are contained in $Q$ has
$$ \HH^{n+m}\left( \bigcup_{R\in S^{i+1}, R\subseteq Q} R \right) < (1-\eta)\HH^{n+m}(Q), $$
for some $\eta=\eta(\epsilon,\delta,n,m)>0$.
\end{lemma}
\begin{proof}
Assume that $Q=Q_0$ for convenience; we can achieve this by simply rescaling. Let $P=P_Q$ as in \eqref{eq:PQdef}.

Let $P_y = (P+y) \cap Q$ for $y\in P^{\bot} \cap [0,1]^m$. Note that each $P_y$ is isometric to the unit cube of $\RR^n$. It will also be convenient to assume that no component of $y$ is of the form $\frac{k}{2^n}$ for any $k,n\in\mathbb{Z}$, so that $\HH^n$-a.e. point of $P_y$ is in at most one dyadic cube of any given scale.

We first give a lower bound on the size of the image of $P_y$.
\begin{claim}\label{claim:lowerbound}
We have
\begin{equation}\label{eq:lowerbound}
 \HH^n_\infty(f(P_y)) \geq (\delta - 2C_0\epsilon)^n. 
\end{equation}
\end{claim}
\begin{proof}[Proof of Claim \ref{claim:lowerbound}]
Observe that if $F$ and $F'$ are any pair of opposite $(n-1)$-dimensional faces of the unit $n$-cube $P_y$, 
then 
$$ \dist(f(F), f(F')) \geq \delta - 2C_0 \epsilon,$$
otherwise a vector $v\in P$ from a point of $F$ to a point of $F'$ would have $\|v\|_Q<\delta \leq \delta|v|$, contradicting the definition \eqref{eq:PQdef} of $P_Q$ for $Q\in S^i\subseteq\QQ_{\text{good}}$. Equation \eqref{eq:lowerbound} then follows from \cite[Corollary 1.6]{Ki16}.
\end{proof}

We next argue an upper bound on the size of the images of $P_y \cap R$, for sub-cubes $R\in S^{i+1}$ that are contained in $Q$.
\begin{claim}\label{claim:upperbound}
Let $R\in S^{i+1}$ be contained in $Q$. There is a constant $C_{n,m}$ depending only on $n$ and $m$ such that
\begin{equation}\label{eq:upperbound}
 \HH^n_\infty(f(P_y \cap R) )\leq C_{n,m}(\delta'+\epsilon)\side(R)^n.
\end{equation}
\end{claim}
\begin{proof}[Proof of Claim \ref{claim:upperbound}]
Since $R\in S^{i+1}$, there is a unit vector $v\in P$ such that
$$\|v\|_R < \delta'.$$

Let $P'$ be the affine $(n-1)$-plane inside $P_y$ that is orthogonal to $v$ and passes through the center of the $n$-dimensional cube $P_y \cap R$. Let $t=  C_0(\delta'\sqrt{n}+\epsilon)$.

We can cover $P'\cap C_0R$, and hence $f(P'\cap C_0R)$, by $\lesssim_{n,m} t^{-(n-1)}$ balls of radius $t\side(R)$. If $x\in P_y\cap R$, then the nearest point $x'$ to $x$ in $P'$ lies in $P'\cap C_0R$, and has $x-x'$ parallel to $v$. Therefore
$$ d(f(x), f(x')) < \|x-x'\|_Q + C_0\epsilon\side(R) < \delta'|x-x'| + C_0\epsilon\side(R)   \leq  t\side(R).$$
Thus, 
$$ f(P_y \cap R) \subseteq N_{ t\side(R)}(f(P' \cap C_0R)) \subseteq \text{ the union of }  \lesssim_{n,m} t^{-(n-1)} \text{ balls of radius } 2t\side(R).$$
Hence
$$ \HH^n_\infty(f(P_y \cap R) )\lesssim_{n,m} t^{-(n-1)} ( 2t\side(R))^n = 2^n t\side(R),$$
which proves \eqref{eq:upperbound}.
\end{proof}

We will need the following basic fact: For each $n\in\mathbb{N}$, there is a constant $\lambda_n>0$ such that, if $R$ is a cube in $\RR^n$, then
\begin{equation}\label{eq:content}
\HH^n_\infty(R) = \lambda_n \side(R)^n.
\end{equation}
Indeed, up to scaling, all such cubes are isometric, so it suffices to understand that the unit cube in $\RR^n$ has positive, finite Hausdorff $n$-content, which is standard.

\medskip

We now decompose $P_y$ as
$$ P_y = \bigcup_{R\in S^{i+1}, R\cap P_y\neq \emptyset} (R\cap P_y) \cup \bigcup_{R'\notin S^{i+1}, R\cap P_y\neq \emptyset} (R\cap P_y),$$
where $R'$ are a collection of almost-disjoint dyadic cubes chosen to cover $P_y\setminus  \cup_{R\in S^{i+1}} (R\cap P_y)$.

Therefore, using Claims \ref{claim:lowerbound} and \ref{claim:upperbound}, and \eqref{eq:content}, we obtain:
\begin{align*}
(\delta-2C_0\epsilon)^n &\leq \HH^n_\infty(f(P_y))\\
&\leq \sum_{R\in S^{i+1}, R\cap P_y\neq \emptyset} \HH^n_\infty(f(R\cap P_y)) + \sum_{R'\notin S^{i+1},  R\cap P_y\neq \emptyset} \HH^n_\infty(f(R'\cap P_y))\\
&\leq C_{n,m}(\delta'+\epsilon)\sum_{R\in S^{i+1}, R\cap P_y\neq\emptyset}\side(R)^n + \lambda_n \sum_{R'\notin S^{i+1},  R\cap P_y\neq \emptyset} \side(R')^n\\
&= C_{n,m}(\delta'+\epsilon)\sum_{R\in S^{i+1}, R\cap P_y\neq\emptyset}\side(R)^n + \lambda_n\left(1- \sum_{R\in S^{i+1}, R\cap P_y\neq\emptyset} \side(R)^n \right)\\
&= \lambda_n - (\lambda_n -  C_{n,m}(\delta'+\epsilon)) \sum_{R\in S^{i+1}, R\cap P_y\neq\emptyset}\side(R)^n
\end{align*}

Rearranging the above inequality yields
$$ \sum_{R\in S^i, R\cap P_y\neq\emptyset}\side(R)^n \leq \frac{\lambda_n - (\delta-2C_0\epsilon)^n}{\lambda_n - C_{n,m}(\delta'+\epsilon)}. $$
Since $\delta'$ and $\epsilon$ are chosen depending on $\delta,n,m$, we may force the fraction on the right hand side to be bounded strictly away from $1$, i.e., so that
\begin{equation}\label{eq:fubini}
\sum_{R\in S^{i+1}, R\cap P_y\neq\emptyset}\side(R)^n \leq 1-\eta, 
\end{equation}
for $\eta = \eta(n,m,\delta,\delta',\epsilon)>0$.

We now apply Fubini's theorem to estimate the $(n+m)$-dimensional volume of the set $\cup_{R\in S^i, R\subseteq Q} R$ by integrating \eqref{eq:fubini} over $y\in [0,1]^{m}$. (Note that the set of dyadic points $y$ that we excluded in proving \eqref{eq:fubini} is a set of $\HH^m$-measure zero.) Writing $|\cdot|$ for Lebesgue measure in $\RR^{n+m}$, we obtain
$$ \left| \cup_{R\in S^{i+1}, R\subseteq Q} R \right| \leq \int_{y\in [0,1]^m} \sum_{R\in S^{i+1}, R\cap P_y\neq \emptyset}\side(R)^n \leq 1-\eta = (1-\eta)|Q|.$$
Therefore
$$ \HH^{n+m}\left( \cup_{R\in S^{i+1}, R\subseteq Q} R \right) \leq (1-\eta)\HH^{n+m}(Q),$$
as desired.
\end{proof}

We now complete the proof of Lemma \ref{lem:packed}.

\begin{proof}[Proof of Lemma \ref{lem:packed}]
By Lemma \ref{lem:packedfubini}, we have

$$ \HH^{n+m}\left(\cup_{Q\in S^{i+1}} Q \right) \leq (1-\eta)\HH^{n+m}\left(\cup_{Q\in S^{i}} Q \right)$$
for each $i\geq 0$, where $\eta>0$ is the constant from Lemma \ref{lem:packedfubini}.

It follows that, choosing $K$ large enough so that $(1-\eta)^K<\alpha'$, we obtain
$$ \HH^{n+m}\left(\cup_{Q\in S^K} Q \right) < \alpha'\HH^{n+m}(Q^0).$$
The choice of $K$ depends therefore on $\alpha'$ and $\eta(n,m,\delta,\delta',\epsilon)$. This completes the proof.
\end{proof}

\subsection{Splitting and definition of the sets}
We continue to work with a fixed $Q^0\in \Delta_k \cap \QQ_{good}$ as in section \ref{sec:stoppingtime}, where $k$ comes from Lemma \ref{lem:initial}.

In section \ref{sec:stoppingtime}, we defined collections of cubes $S^0, S^1, S^2, \dots \subseteq \QQ_{\text{good}} \subseteq \Delta$. Each collection $S^i$ consists of pairwise disjoint dyadic cubes contained in $Q^0$, and each cube of $S^{i+1}$ is contained a cube of $S^i$.

In this subsection, we split each $S^i$ into a controlled number of disjoint sub-collections $S^i_j$, using Lemma \ref{lem:splitting}. We will then use the $S^i_j$ to define the sets required by Proposition \ref{prop:hbilip}.

For the first step, fix an odd integer $\Lambda>3$, which will depend on $\delta'$, $n$, and $m$, and let $\alpha'' = \frac{\alpha'}{K}$, which depends on $\alpha$ and the constant $K=K(n,m,\delta, \delta',\epsilon)$ from Lemma \ref{lem:packed}.

Apply Lemma \ref{lem:splitting} to each collection $S^i$, with the parameter $\eta$ set as $\alpha''$ and expansion constant $\Lambda$.

This partitions each $S^i$ into a controlled number of families $S^i_j$ ($1\leq j \leq k_i \leq k_0(\alpha'',n,m)$) that are $\Lambda$-disjoint, and a garbage set $G^i\subseteq S^i$ whose total $\HH^{n+m}$-measure is less than $\alpha''$.

Consider the collection of ``words'' $w$ taken from the set
\begin{equation}\label{eq:Wdef}
 W = \bigcup_{\ell=0}^{K} \prod_{i=1}^{\ell} \{1, 2, \dots, k_i\}. 
\end{equation}
In other words, a word $w\in W$ is of the form $(j_0, j_1, j_2, \dots, j_\ell)$, where $0\leq \ell \leq K$ and $1\leq j_i \leq k_i$ for each $i\in\{1,\dots, \ell\}$. (The number of words in $W$ is controlled based only on $K$ and $k_0$, hence independently of $Q^0$.)

Fix a word $w=(j_0, j_1, j_2, \dots, j_\ell)\in W$. This yields collections of cubes:
\begin{equation}\label{eq:T0w}
 T^0_w = \{Q^0\} 
\end{equation}
\begin{equation}\label{eq:Tiw}
 T^i_w = \{Q\in S^i_{j_i} :  Q\subseteq R \text{ for some } R \text{ in } T^{i-1}_w \}.
\end{equation}

These collections have the following properties for $i\leq \ell$:
\begin{itemize}
\item Each $T^i_w$ is a sub-collection of $S^i_{j_i} \subseteq S^i$.
\item If $Q,Q'\in T^i_w$, then $\Lambda Q \cap \Lambda Q' = \emptyset.$
\item If $i<\ell$ and $Q\in T^{i+1}_w$, then $Q$ is contained in a cube of $T^{i}_w$.
\item If $Q\in S^i \setminus G^i$, then $Q\in T^i_w$ for some $w\in W$.
\end{itemize}

Independently of the above constructions, we may now also apply Lemma \ref{lem:coding} with parameter $\eta=\alpha'$. This yields a decomposition
\begin{equation}\label{eq:codingapplication}
 Q_0 = A_1 \cup A_2 \cup \dots \cup A_{M_{\md}} \cup G_{\md},
\end{equation}
where $A_i$ have the property described in Lemma \ref{lem:coding}, the number $M_{\md}$ depends on $\alpha'$, $\epsilon$, and $n+m$, and $\HH^{n+m}_\infty(G_{\md})<\alpha'$.

We are now ready to define the sets appearing in Proposition \ref{prop:hbilip}.

\begin{definition}[Definition of the sets in Proposition \ref{prop:hbilip}]\label{def:F}
Let $k$ be as in Lemma \ref{lem:initial} and let $Q^0\in \QQ_{\text{good}} \cap \Delta_k$. Let $M_{\md}$ be as in Lemma \ref{lem:coding} and let $p\in \{1, \dots, M_{\md}\}$. Let $w=(j_0, \dots, j_\ell)\in W$ as defined in \eqref{eq:Wdef}.

We define a set $F(Q^0, p, w) \subseteq Q^0$ as follows:
$$ F(Q^0,p,w) =  A_p \cap  \left[\bigcap_{i=1}^\ell \left(\bigcup_{Q\in T^i_w} Q\right)   \setminus \left( \bigcup_{Q\in S^{\ell+1}} Q \cup \bigcup_{Q\in\QQ_{\text{compressed}}} (3Q)\right)\right].$$
Here $T^i_w$ refer to the cube collections defined in \eqref{eq:T0w} and \eqref{eq:Tiw}.
\end{definition}

Note that $F(Q^0,p,w)\subseteq Q^0$, and its dependence on $Q^0$ is implicit in our construction of the sets $S^i$ and $T^i_w$.

To clarify Definition \ref{def:F}, a point $x\in Q^0$ is in $F(Q^0, p, w)$ if and only if the following conditions hold:
\begin{enumerate}[(i)]
\item The point $x$ is in the set $A_p$ from our use of Lemma \ref{lem:coding} in \eqref{eq:codingapplication}.
\item For each $0\leq i \leq \ell$, $x$ is contained in a cube of the collection $T_w^i$, where $T_w^0 = \{Q_0\}$.
\item The point $x$ is not contained in any cube of $S^{\ell+1}$. (Hence, $x$ is contained in exactly $\ell+1$ cubes of $S$, those in (ii).)
\item The point $x$ is not contained in the triple of any cube of $\QQ_{\text{compressed}}$.
\end{enumerate}

These sets $F(Q^0,p,w)$, for each choice of ``top cube'' $Q^0$ provided by Lemma \ref{lem:initial}, each $p\in\{1,\dots, M_{\md}\}$, and each $w\in W$, will comprise the sets (called $F_i$) in Proposition \ref{prop:hbilip}. The following lemma reflects the fact that there are a controlled number of such sets, and that property (ii) of Proposition \ref{prop:hbilip} holds.

\begin{lemma}\label{lem:leftover}
Let $k$ be as in Lemma \ref{lem:initial}. As $Q^0$ ranges over all cubes in $\QQ_{good} \cap\Delta_k$, $p$ ranges from $1$ to $M_{\md}$, and $w\in W$ the number of sets $F(Q^0,p,w)$ is controlled by a constant depending only on $n,m,\alpha,$ and our previously defined constants $\delta,\delta',\epsilon, K, K_1, k_0$.

Furthermore, we have
\begin{equation}\label{eq:leftover}
 \HH^{n,m}_\infty\left(f, Q_0 \setminus \bigcup_{Q^0 \in\QQ_{\text{good}} \cap\Delta_k} \bigcup_{p=1}^{M_{\md}} \bigcup_{w\in W} F(Q^0,p,w) \right) < \alpha.
\end{equation}
\end{lemma}
\begin{proof}
The number of choices of $Q^0 \in \QQ_{good} \cap\Delta_k$ is controlled by $|\Delta_k|=2^{(n+m)k} \leq 2^{(n+m)K_1}$, where $K_1$ is the constant from Lemma \ref{lem:initial}. The number of choices for $p\in \{1,\dots,M_{\md}\}$ is controlled by $M_{\md}$, which depends only on $\epsilon$ and $n+m$. Lastly, the number of choices for $w\in W$ is bounded by a constant depending on the constants $K$ from Lemma \ref{lem:packed} and the constant $k_0$ arising from Lemma \ref{lem:splitting} with parameters depending on $\alpha',\delta, n,m$. 

It follows that the number of sets $F(Q^0,p,w)$ is controlled as desired.

We now focus on bounding the mapping content of the remaining points that are not in any set $F(Q^0,p,w)$. This is just a matter of assembling some prior results.

If a set $x\in Q_0$ is not in $\bigcup_{Q^0 \in \QQ_{\text{good}}  \cap\Delta_k} \bigcup_{F_w\subseteq Q^0} F_w$, then there are a few options:
\begin{enumerate}
\item\label{eq:bad1} $x$ may not be in any of the sets $A_1, \dots A_{M_{\md}}$ given by Lemma \ref{lem:coding} with parameter $\eta=\alpha'$.
\item\label{eq:bad2} $x$ may be in a cube $Q$ of $\Delta_k \setminus \QQ_{\text{good}} $,
\item\label{eq:bad3} $x$ may be in one of the garbage sets $G^i$, for $i\in \{1\dots K\}$,
\item\label{eq:bad4} $x$ may be in a cube of $S^{K+1}$, where $K$ is the constant from Lemma \ref{lem:packed}.
\item\label{eq:bad5} $x$ may be in a triple of a cube of $\QQ_{\text{compressed}}$.
\end{enumerate}

The set of points satisfying \eqref{eq:bad1} has total $\HH^{n+m}_\infty$ less than $\alpha'$, by Lemma \ref{lem:coding}.

The set of points satisfying \eqref{eq:bad2} has total $\HH^{n,m}_\infty(f,\cdot)$ less than $\alpha'$, by Lemma \ref{lem:initial}.

The set of points satisfying \eqref{eq:bad3} has total $\HH^{n+m}$-measure less than $K\alpha'' = \alpha'$, since there are at most $K$ sets $G^i$ and each has $\HH^{n+m}$-measure less than $\alpha''$ by our use of Lemma \ref{lem:splitting}.

The set of points satisfying \eqref{eq:bad4} has total $\HH^{n+m}$-measure less than $\alpha'$, by Lemma \ref{lem:packed}.

Lastly, the set of points satisfying \eqref{eq:bad5} has total  $\HH^{n,m}_\infty(f,\cdot)$  less than $C(\epsilon+\delta)$, by Lemma \ref{lem:compressedcontent}, where $C$ depends on $n$ and $m$.

Thus by ensuring that $\delta$ and $\epsilon$ are small depending only on $\alpha$, $n$, and $m$, recalling the definition of $\alpha'$ from \eqref{eq:alpha'def}, and using Lemmas \ref{lem:subadditive} and  \ref{lem:areaboundscontent}, we obtain \eqref{eq:leftover}.

\end{proof}

\subsection{Construction of the bi-Lipschitz mappings}

For the remainder of this section, we fix $k$ as in Lemma \ref{lem:initial}, $Q^0\in \QQ_{good} \cap\Delta_k$, $p\in\{ 1, \dots, M_{\md}\}$, and $w=(j_0, \dots, j_\ell)\in W$ as in \eqref{eq:Wdef}. 

We then obtain a set $F=F(Q^0,p,w)$ as in Definition \ref{def:F}, which will be fixed for the remainder of the section. We will define a bi-Lipschitz mapping $\phi: F \rightarrow \phi(F)\subseteq Q_0$. The collection of such mappings $\phi$ for all choices of $F$ will comprise the mappings called $\phi_i$ in Proposition \ref{prop:hbilip}.

Recall the cube collections $T^i_w$ defined in \eqref{eq:T0w} and \eqref{eq:Tiw}. By definition of $F$, we have
$$ F \subseteq \cap_{i=0}^{\ell} \left(\cup_{Q\in T^i_w} Q \right). $$

As noted below \eqref{eq:Tiw}, the collections $T^0_w, \dots, T^\ell_w$ satisfy:
\begin{itemize}
\item Each $T^i_w$ is a sub-collection of $S^i_{j_i} \subseteq S^i$.
\item If $Q,Q'\in T^i_w$, then $\Lambda Q \cap \Lambda Q' = \emptyset.$
\item If $Q\in T^{i+1}_w$, then $Q$ is contained in a cube of $T^{i}_w$.
\end{itemize}

Recall that each cube $Q$ in any $T^i_w$ is also, by definition, an element of $\QQ_{\text{good}}$. Thus, there is a seminorm $\|\cdot\|_Q$ satisfying
$$ \sup_{x,y\in C_0 Q} | d(f(x), f(y)) - \|x-y\|_Q | \leq \md(C_0Q)\side(C_0 Q) < C_0\epsilon\side(Q).$$
Moreover, as noted in \eqref{eq:PQdef}, there is a coordinate $n$-plane $P_Q\in\PP$ such that
$$ \|v\|_{Q} \geq \delta \text{ for all unit vectors } v\in P_Q.$$

For each $1\leq i \leq \ell$, let $\phi^i: \left(\cup_{Q\in T^i_w} Q\right) \rightarrow Q_0$ be a map defined as follows: On each cube $Q \in T^i_w$, the restriction $\phi^i|_{Q}$ of $\phi^i$ to $Q$ is an affine map $A_{Q}x+b_{Q}$ such that $A_{Q}$ is linear and orthogonal,
$$ \phi^i|_{Q}({Q}) = Q$$ 
and
$$A_{Q}(P_{Q}) = P_{Q'},$$
where $Q'$ is the unique cube of $T^{i-1}_w$ containing $Q$. Note that $P_Q$ and $P_{Q'}$ are \textit{coordinate} $n$-planes (i.e., elements of $\PP$), so such affine mappings exist.

Similarly, we define $\phi^0:Q^0 \rightarrow Q_0$ to be an affine map $Ax+b$, with $A$ orthogonal, such that
$$ \phi^0(Q^0)=Q^0$$
and
$$ A(P_{Q^0}) = P_0.$$
Recall that $P_0$ denotes the span of the first $n$ standard basis vectors in $\RR^{n+m}$.

Thus, each $\phi^i$ ``rotates in place'' the cubes $Q\in T^i_w$. The fact that $\Lambda>3$ and $\Lambda Q \cap \Lambda Q' = \emptyset$ for each $Q,Q'\in T^i_w$ implies that $\phi^i$ is well-defined and bi-Lipschitz on the set $\cup_{Q\in T^i_w} Q \supseteq F$, with an absolute bi-Lipschitz constant. Moreover, $\phi^i(\cup_{Q\in T^i_w} Q) = \cup_{Q\in T^i_w} Q$, since $\phi^i$ fixes each cube of $T^i_w$ setwise.

Therefore we can define $\phi : F \rightarrow Q_0 $ by
\begin{equation}\label{eq:phidef}
\phi = \phi^0 \circ \phi^1 \circ \phi^2 \circ \dots \circ \phi^\ell.
\end{equation}

Since each $\phi^i$ is bi-Lipschitz on $F$ with absolute constant, the map $\phi$ is bi-Lipschitz on $F$ with constant depending only on $K$, hence only on $\alpha,n,m,\delta,\delta',\epsilon$.

Informally, $\phi$ acts on $F$ as a sort of ``clockwork mechanism''. It rotates small scale cubes, then larger scale cubes, etc. with the goal of ``lining up'' the planes $P_Q$ to match $P_0$. We will see this in the next section.

\subsection{Conclusion of the proof of Proposition \ref{prop:hbilip}}
We continue to fix the set $F=F(Q^0,p,w)$ and the map $\phi$ defined in \eqref{eq:phidef}.

To complete the proof of Proposition \ref{prop:hbilip}, we will show that the map
$$ \tilde{h} = (f\circ \phi^{-1}, \pi_{P_0^{\bot}})$$
is (quantitatively) bi-Lipschitz on $\phi(F)$. Since all the component functions are separately Lipschitz, with bounds depending only on our chosen constants, it is immediate that $\tilde{h}$ is itself Lipschitz, with constant depending only on $\alpha$, $n$, $m$ (and our previously chosen constants, which ultimately will only depend on these). The remaining work will be to show the lower bound.

Let $x$ and $y$ be distinct points of $F$, so $\phi(x), \phi(y) \in \phi(F)$. Let $Q$ be a cube of minimal side length such that $3 Q$ contains both $x$ and $y$. Since $x,y\in F\subseteq Q^0$, we may take $Q\subseteq Q^0$. Observe that
\begin{equation}\label{eq:C'def}
\frac{\diam(3Q)}{|x-y|} \leq \frac{\diam(C_0 Q)}{|x-y|} \leq C'_{n,m}
\end{equation}
where $C'_{n,m}$ is a dimensional constant depending only on $n$ and $m$.

Note also that by definition of $F$, we must have that
$$ Q\in \QQ_{\text{good}}.$$
Indeed, no points of $F$ can be in a triple of a cube in $\QQ_{\text{compressed}}$, and the fact that $x,y\in F \subseteq A_p$ guarantees that $Q\notin \QQ_{\md}$ by Lemma \ref{lem:coding}. Thus, $Q\in \QQ_{\text{good}}$.

Let $i_0\in\{0, 1, \dots, \ell\}$ be the largest index such that $Q\subseteq Q_{i_0}$ for some $Q_{i_0}\in T_w^{i_0}$.

We begin with the following:

\begin{lemma}\label{lem:projectionbound}
With $x$, $y$, $i_0$, and $Q_{i_0}$ as above, we have
$$|\pi_{P_0^{\bot}}(\phi(x)) - \pi_{P_0^\bot}(\phi(y))| \geq |\pi_{P_{Q_{i_0}}^{\bot}}(x) - \pi_{P_{Q_{i_0}}^{\bot}}(y)| - \frac{2C'_{n,m}}{\Lambda}|x-y|. $$ 
Here $C'_{n,m}$ is the dimensional constant defined in \eqref{eq:C'def}.
\end{lemma}
\begin{proof}
Note that the map
$$ \psi = \phi^0 \circ \phi^1 \circ \phi^2 \circ \dots \circ \phi^{i_0}, $$
is simply an affine map of full rank when restricted to $Q_{i_0}$.

By definition of $\phi^0$, $\phi^1$, $\phi^2$, etc., we have
$$ (\psi|_{Q_{i_0}})(P_{Q_{i_0}}) \text{ is parallel to } P_0$$
and therefore
$$ (\psi|_{Q_{i_0}})(P_{Q_{i_0}}^{\bot}) \text{ is parallel to } P_0^{\bot}. $$

By standard linear algebra, there is therefore an isometry $\iota\colon P_{Q_{i_0}}^\bot \rightarrow P_0^{\bot}$ such that
$$ \pi_{P_0^{\bot}} \circ (\psi|_{Q_{i_0}}) = \iota \circ \pi_{P_{Q_{i_0}}^{\bot}}.$$

We now consider two cases: $i_0=\ell$ or $i_0 < \ell$.

If $i_0=\ell$, then $\psi=\phi$ and it follows that
\begin{align*}
|\pi_{P_0^{\bot}}(\phi(x)) - \pi_{P_0^\bot}(\phi(y))| &= | \iota \circ \pi_{P^\bot_{Q_{i_0}}}(x) -  \iota \circ \pi_{P^\bot_{Q_{i_0}}} (y)|\\
&= | \pi_{P^\bot_{Q_{i_0}}}(x) - \pi_{P^\bot_{Q_{i_0}}} (y)|,
\end{align*}
which proves the lemma in this case.

If $i_0<\ell$, then let $\psi' = \phi^{i_0+1} \circ \dots \circ \phi^{\ell}$, so that $\phi = \psi \circ \psi'$. Then
\begin{align}
|\pi_{P_0^{\bot}}(\phi(x)) - \pi_{P_0^\bot}(\phi(y))| &=  | \iota \circ \pi_{P_{Q_{i_0}}^{\bot}} \circ \psi'(x) -  \iota \circ \pi_{P_{Q_{i_0}}^{\bot}} \circ \psi'(y)|\\
&=  |  \pi_{P^\bot_{Q_{i_0}}} \circ \psi'(x) -   \pi_{P^\bot_{Q_{i_0}}} \circ \psi'(y)| \label{eq:pi1}
\end{align}
Now, $x$ and $y$ are in different cubes $Q^x_{i_0+1}$ and $Q^y_{i_0+1}$ of $T^{i_0+1}_w$ (by maximality of $i_0$). Since $\psi'$ only rotates cubes at levels $i_0+1$, $i_0+2$, etc., we see that
$$\psi'(x)\in Q^x_{i_0+1} \text{ and } \psi'(y)\in Q^y_{i_0+1}.$$
Moreover, $Q^x_{i_0+1}$ and $Q^y_{i_0+1}$ are contained in $3Q$, and $\Lambda Q^x_{i_0+1}$ and $\Lambda Q^y_{i_0+1}$ are disjoint. Therefore,
$$ |\psi'(x) - x| \leq \diam(Q^x_{i_0+1}) \leq \frac{1}{\Lambda}\diam(\Lambda Q^x_{i_0+1}) \leq \frac{1}{\Lambda} \diam(3Q) \leq \frac{C'_{n,m}}{\Lambda}|x-y|,$$ 
with $C'_{n,m}$ as in \eqref{eq:C'def}. The analogous statement holds for $|\psi'(y)-y|$.

Therefore, using \eqref{eq:pi1},
\begin{align*}
|\pi_{P_0^{\bot}}(\phi(x)) - \pi_{P_0^\bot}(\phi(y))|  &= |  \pi_{P_{Q_{i_0}}^{\bot}} \circ \psi'(x) -   \pi_{P_{Q_{i_0}}^{\bot}} \circ \psi'(y)|\\
&\geq |  \pi_{P_{Q_{i_0}}^{\bot}} (x) -   \pi_{P_{Q_{i_0}}^{\bot}} (y)| - |\psi'(x) - x| - |\psi'(y)-y|\\
&\geq |  \pi_{P_{Q_{i_0}}^{\bot}} (x) -   \pi_{P_{Q_{i_0}}^{\bot}} (y)| - \frac{2C'_{n,m}}{\Lambda} |x-y|.
\end{align*}

This proves Lemma \ref{lem:projectionbound}.
\end{proof}

As $Q\in\QQ_{\text{good}}$, it has a good approximating seminorm for $f$, namely $\|\cdot\|_Q$. By the stopping time construction of the sets $S^i$, we must have that
$$\|v\|_Q \geq \delta'|v|$$
for all $v\in P_{Q_{i_0}}$. (If this did not hold, then $Q$ would have been added to $S^{i_0+1}$, violating the maximality of $i_0$.)

Write $w=x-y=w_1+w_2$, where $w_1\in P_{Q_{i_0}}$ and $w_2\in P_{Q_{i_0}}^\bot$. Then, using Lemma \ref{lem:normbound},
$$ \|w\|_Q + |w_2| \geq \|w_1\|_Q - \|w_2\|_Q + |w_2| \geq \delta'|w_1| - C_0\epsilon\side(Q) \geq \delta'|w_1| - C'_{n,m}\epsilon|w|$$
and, trivially,
$$ \|w\|_Q + |w_2| \geq |w_2|, $$
from which it follows that
\begin{equation}\label{eq:wnormbound}
 \|w\|_Q + |w_2| \gtrsim_{n,m} \delta'|w| - C'_{n,m}\epsilon|w|,
\end{equation}
with the implied constant depending only on $n$ and $m$. 

Using Lemma \ref{lem:projectionbound} and \eqref{eq:wnormbound}, we obtain
\begin{align*}
d(\tilde{h}(\phi(x)) , \tilde{h}(\phi(y))) &\gtrsim d(f(x) , f(y)) + |\pi_{P_0}(\phi(x)) - \pi_{P_0}(\phi(y)))|\\
&\geq \|x-y\|_Q - C_0\epsilon\side(Q)  + |\pi_{P_{Q_{i_0}}^{\bot}}(x) - \pi_{P_{Q_{i_0}}^{\bot}}(y)| - \frac{2C'_{n,m}}{\Lambda}|x-y|\\
&\geq \|w\|_Q - C_0 C'_{n,m}\epsilon|w|  + |w_2| -  \frac{2C'_{n,m}}{\Lambda}|w|\\
&\gtrsim_{n,m} \delta'|w| - \left(C_0 C'_{n,m}\epsilon + \frac{2C_{n,m}}{\Lambda}\right)|w|\\
&\gtrsim |x-y|\\
&\gtrsim |\phi(x) - \phi(y)|
\end{align*}
The penultimate line follows by taking $\epsilon>0$ sufficiently small and $\Lambda>3$ sufficiently large, depending on $n,m,\delta'$. The final line follows from the fact that $\phi$ is bi-Lipschitz.

This verifies property (i) of Proposition \ref{prop:hbilip}, and so along with Lemma \ref{lem:leftover} completes the proof of this proposition.

\section{Proof of Theorem \ref{thm:hardersarder}}\label{sec:mainproof}

We now combine Propositions \ref{prop:directional} and \ref{prop:hbilip} to prove Theorem \ref{thm:hardersarder}.

\begin{proof}[Proof of Theorem \ref{thm:hardersarder}]
Let $f\colon Q_0\rightarrow X$ satisfy the assumptions of Theorem \ref{thm:hardersarder}. Fix $\gamma>0$. 

Without loss of generality, we may assume that
$$\HH^n(f(Q_0)) \geq \HH^n_\infty(f(Q_0)) >0,$$
otherwise the conclusion of the theorem is trivial by Lemma \ref{lem:easymapping}.

By Proposition \ref{prop:hbilip}, we may decompose $Q_0$ as
$$ Q_0 = F_1 \cup F_2 \cup \dots \cup F_N \cup G, $$
where
\begin{equation}\label{eq:hbilipgarbage}
 \HH^{n,m}_\infty(f,G) < \gamma/2 
\end{equation}
and, for each $i\in\{1,\dots,N\}$, there is a bi-Lipschitz map $\phi_i\colon F_i\rightarrow Q_0$
such that the map
\begin{equation}\label{eq:htilde}
 (f\circ \phi_i^{-1}, \pi_{P_0^{\bot}}),
\end{equation}
i.e., the map
$$ (x,y) \mapsto (f(\phi_i^{-1}(x,y)), y),$$
is bi-lipschitz on $\phi_i(F_i)$.

Apply Corollary \ref{cor:bilipextension} to each map $\phi_i$ on $F_i$ with parameter $\kappa = \gamma/(10C_{n,m} N)$, where $C_{n,m}$ is the implied constant from Lemma \ref{lem:areaboundscontent}. This decomposes each set $F_i$ into a controlled number of sets 
$$ F_i = F_{i,1} \cup \dots \cup F_{i,M_i} \cup G_i$$
such that each map $\phi_i|_{F_{i,j}}$ admits a globally bi-Lipschitz extension 
$$ \phi_{i,j} : \RR^{n+m} \rightarrow \RR^{n+m}$$
and
\begin{equation}\label{eq:phigarbage}
\HH^{n,m}_\infty(f, G_i) \leq C_{n,m}\HH^{n+m}(G_i) < \frac{\gamma}{10N} \text{ for each } i\in\{1,\dots,N\},
\end{equation}
using Lemma \ref{lem:areaboundscontent}.

Note that the bi-Lipschitz constants of $\phi_{i,j}$ and the constants $M_i$ are controlled depending only on $\gamma$, $n$, and $m$. By passing to the closure, we may assume without loss of generality that the sets $F_{i,j}$ are each compact.

Write $\tilde{f}_{i,j} = f \circ \phi_{i,j}^{-1} \colon \phi_{i,j}(F_{i,j}) \rightarrow X$. The fact that the map in \eqref{eq:htilde} is bi-Lipschitz on $\phi_{i,j}(F_{i,j})$ allows us to apply Proposition \ref{prop:directional} to the map $\tilde{f}_{i,j}$ on the set $\phi_{i,j}(F_{i,j})$. (Note that, while $\tilde{f}_{i,j}$ is not necessarily $1$-Lipschitz, its Lipschitz constant is controlled depending on $n$,$m$, $\gamma$.)

We thus apply Proposition \ref{prop:directional} to the map $\tilde{f}_{i,j}$ on the set $\phi_{i,j}(F_{i,j})$ with parameter $\alpha=\alpha_{i,j}>0$ sufficiently small, depending on $n$, $m$, $\gamma$, and the bi-Lipschitz constant of $\phi_{i,j}$ (and thus ultimately only on $n$, $m$, and $\gamma$).

This decomposes each set $\phi_{i,j}(F_{i,j})$ into
$$ \phi_{i,j}(F_{i,j}) = \tilde{E}_{i,j}^1 \cup \tilde{E}_{i,j}^2 \cup \dots \cup \tilde{E}_{i,j}^{M_{i,j}} \cup \tilde{G}_{i,j},$$
where each $\tilde{E}_{i,j}^k$ is a Hard Sard set for $\tilde{f}_{i,j}$ and 
\begin{equation}
\HH^{n+m}(\tilde{G}_{i,j}) < \alpha_{i,j}.
\end{equation}
In particular,
\begin{equation}\label{eq:ggarbage}
\HH^{n,m}_\infty(f, \phi_{i,j}^{-1}(\tilde{G}_{i,j})) \leq \HH^{n+m}(\phi_{i,j}^{-1}(\tilde{G}_{i,j})) < \frac{\gamma}{10NM_i},
\end{equation}
if we choose $\alpha_{i,j}$ sufficiently small, using Lemma \ref{lem:areaboundscontent} and the fact that $\phi_{i,j}$ is quantitatively bi-Lipschitz.

Note that $M_{i,j}$ and the Hard Sard constants $C_{Lip}$ for $\tilde{E}_{i,j}^k$ are controlled depending only on $\gamma$, $N$, and $L$, and hence only on $\gamma$. Let $g_{i,j}^k$ denote the globally $C_{Lip}$-bi-Lipschitz Hard Sard map for $\tilde{f}_{i,j}$ on the set $\tilde{E}_{i,j}^k$.

Recall that, because of condition \eqref{eq:gshear} in Proposition \ref{prop:directional}, we know that $g_{i,j}^k|_{\tilde{E}^k_{i,j}}$ is a ``shear'' that preserves the $y$-coordinate. In particular, if $(x,y)\in\tilde{E}^k_{i,j}$, then
\begin{equation}\label{eq:gshear2}
g^k_{i,j}(x,y) = (x',y) \text{ for some } x'\in\RR^n.
\end{equation}

Let $E_{i,j}^k = \phi_{i,j}^{-1}( \tilde{E}_{i,j}^k)$. Note that the number of these sets was controlled in each step, depending ultimately only on $n$, $m$, and $\gamma$.

We now claim that each pair $(E_{i,j}^k, g_{i,j}^k \circ \phi_{i,j})$ is a Hard Sard pair for our original map $f$. Let us fix indices $i,j,k$ as above and call this pair $(E, g\circ \phi)$.

First of all, Condition \eqref{HS2} of Definition \ref{def:HSpair} holds, because $g$ and $\phi$ were constructed above as globally defined bi-Lipschitz mappings with quantitative constants on $\RR^{n+m}$.

We next verify  that Condition \eqref{HS3} of Definition \ref{def:HSpair} holds for the pair $(E,g\circ\phi)$. Fix $(x,y)$ and $(x',y')$ in $g\circ \phi(E)$. By construction $(\phi(E), g)$ is a Hard Sard pair for $f\circ \phi^{-1}$. Therefore,
$$ x=x' \Leftrightarrow f\circ \phi^{-1} \circ g^{-1}(x,y) = f\circ \phi^{-1} \circ g^{-1}(x',y')\Leftrightarrow f \circ (g\circ \phi)^{-1}(x,y) = f \circ (g\circ \phi)^{-1}(x',y'),$$
which exactly verifies Condition \eqref{HS3} for the Hard Sard pair $(E,g\circ \phi)$ for $f$.

Lastly, we verify that Condition \eqref{HS4} of Definition \ref{def:HSpair} holds for the pair $(E,g\circ\phi)$. In other words, we verify that the mapping
$$ (x,y) \mapsto (f\circ(g\circ \phi)^{-1}(x,y),y)$$
is bi-Lipschitz on $g(\phi(E))$, with quantitative control on the bi-Lipschitz constant. We will again use the fact that $(\phi(E), g)$ is a Hard Sard pair for $f\circ \phi^{-1}$, as well as property \eqref{eq:gshear2} of $g$.

Consider two points $(x_1, y_1)$ and $(x_2, y_2)$ in $g(\phi(E))$. Using \eqref{eq:gshear2}, we can write
$$ g^{-1}(x_i, y_i) = (x'_i, y_i) \text{ for } i=1,2.$$
Then
\begin{align*}
d\left( \left( f\circ(g\circ \phi)^{-1}(x_1, y_1),y_1\right), \left(f\circ(g\circ \phi)^{-1}(x_2, y_2),y_2\right) \right)&=d\left( \left(f\circ \phi^{-1}(x'_1, y_1),y_1\right), \left(f\circ \phi^{-1}(x'_2, y_2),y_2\right) \right)\\ 
&\approx |(x'_1, y_1) - (x'_2, y_2)|\\
&\approx |(x_1, y_1) - (x_2,y_2)|.
\end{align*}
The second line is because $(\phi(E),g)$ is a Hard Sard pair for $f\circ \phi^{-1}$, and the third is because $g$ is bi-Lipschitz. All implied constants depend only on $n$, $m$, and the Hard Sard constant associated to $(\phi(E),g)$, which we controlled depending on $n$,$m$, and $\gamma$. This verifies Condition \eqref{HS4} of Definition \ref{def:HSpair} for  $(E,g\circ\phi)$.

Thus, each pair $(E_{i,j}^k, g_{i,j}^k \circ \phi_{i,j})$ defined above is a Hard Sard set for $f$. From our work above, all Hard Sard constants $C_{Lip}$ and the total number of these sets is controlled, depending only on $\gamma$, $n$, and $m$.

To prove Theorem \ref{thm:hardersarder}, it remains to control the size of the ``garbage set'' $Q_0 \setminus \cup E_{i,j}^k$. This set can be written as
$$ \hat{G} = G \cup \bigcup_{i=1}^N G_i \cup \bigcup_{i=1}^N \bigcup_{j=1}^{M_i}  \phi_{i,j}^{-1}(\tilde{G}_{i,j}).$$
Thus, we have
$$\HH^{n,m}_\infty(f,\hat{G}) < \gamma$$
using equations \eqref{eq:hbilipgarbage}, \eqref{eq:phigarbage}, \eqref{eq:ggarbage}, and Lemma \ref{lem:subadditive}.

This completes the proof of Theorem \ref{thm:hardersarder}.

\end{proof}

\section{Proof of Theorem \ref{thm:onedim}}\label{sec:onedim}
In this section, we prove Theorem \ref{thm:onedim}.

\begin{proof}[Proof of Theorem \ref{thm:onedim}]
Let $Q_0 = [0,1]^{1+m}$ and let $f\colon Q_0 \rightarrow X$ be a $1$-Lipschitz map into a metric space $X$.

Suppose that $\HH^{1,m}(f,Q_0)<\eta$ for some $\eta>0$. We will show that $\diam(f(Q_0))<\eta'$, for some constant $\eta'$ depending only on $\eta$ and $m$ and tending to $0$ as $\eta$ tends to $0$. This suffices to prove Theorem \ref{thm:onedim}.

The fact that  $\HH^{1,m}(f,Q_0)<\eta$ means that there is a cover $\mathcal{Q}$ of $Q_0$ by almost-disjoint dyadic cubes satisfying
$$ \eta > \sum_{Q\in\QQ} \HH^1_\infty(f(Q))\side(Q)^m \approx \sum_{Q\in\QQ} \diam(f(Q))\side(Q)^m.$$
In the equality, we use the fact that for compact, connected sets (like $f(Q)$), one-dimensional Hausdorff content is comparable to diameter. 

If $m=0$, the result holds, since
$$ \diam(f(Q_0)) \lesssim \HH^1_\infty (f(Q_0)) \leq \sum_{Q\in\QQ} \HH^1_\infty(f(Q)) < \eta.$$

For the remainder of the proof, we therefore assume $m\geq 1$.

Fix $\delta>0$, to be specified later, and let
$$\QQ_\delta = \{Q\in\QQ: \diam f(Q) \leq \delta \side(Q)\}.$$
Let $A$ be the union of all cubes in $\QQ\setminus \QQ_\delta$. Then
\begin{equation}\label{eq:Asize}
 \HH^{1+m}(A) = \HH^{1+m}\left( \bigcup_{Q\in \QQ\setminus\QQ_\delta} Q \right) \leq \sum_{Q\in\QQ\setminus\QQ_\delta} \delta^{-1} \diam(f(Q))\side(Q)^m \lesssim \frac{\eta}{\delta}.
\end{equation}

For each $i\in \{1,\dots,1+m\}$ and $y=(y_1, \dots, y_m) \in [0,1]^m$, consider the line segment
$$  L_i^y = \{ (y_1, y_2, \dots, y_{i-1}, t, y_i, \dots, y_m) : t\in [0,1]\} \subseteq Q_0.$$
In other words, $L_i^y$ is simply the line segment in coordinate direction $i$ emitting from point $y$ in the appropriate orthogonal $m$-plane.

By \eqref{eq:Asize} and Fubini's theorem, for each $i\in \{1, \dots, 1+m\}$, we have 
$$ \frac{\eta}{\delta} \gtrsim |A| = \int_{[0,1]^m} \HH^1(L_i^y \cap A) \,dy$$

Fix $s\in (0,1)$, also to be specified below. There is therefore a Borel set $K_i \subseteq [0,1]^m$ and an absolute constant $c$ such that
\begin{equation}\label{eq:Ki}
 |K_i| > 1-c\frac{\eta}{\delta s}
\end{equation}
and
$$ \HH^{1}(L_i^y \cap A) \leq s \text{ for all } y\in K_i.$$
Without affecting these bounds, we may also assume for technical convenience that no coordinate of any $y\in K_i$ is dyadic, i.e., that if $y\in K_i$ then $y\neq a2^{-b}$ for any integers $a$ and $b$. This removes only a set of measure zero from $K_i$.

With this assumption, we see that, for each $i\in\{1, \dots, 1+m\}$ and $y\in K_i$, the line $L_i^y$ contains no non-trivial segment that lies on the boundary of a dyadic cube. Hence, if $y\in K_i$,
$$ 1 = \length(L_i^y) = \sum_{Q\in\QQ, Q\cap L_i^y\neq \emptyset} \diam(Q)$$
and
$$ s \geq \length(L_i^y \cap A) = \sum_{Q\in\QQ\setminus \QQ_\delta, Q\cap L_i^y\neq \emptyset} \diam(Q).$$

Given $y\in K_i$, we can therefore write
\begin{align}
\diam(f(L_i^y)) &\leq \sum_{Q\in \QQ, Q\cap L_i^y \neq \emptyset} \diam(f(L_i^y \cap Q))\\
&= \sum_{Q\in \QQ_\delta, Q\cap L_i^y \neq \emptyset} \diam(f(L_i^y \cap Q)) + \sum_{Q\in\QQ\setminus \QQ_\delta, Q\cap L_i^y \neq \emptyset} \diam(f(L_i^y \cap Q))\\
&\leq   \sum_{Q\in \QQ_\delta, Q\cap L_i^y \neq \emptyset} \delta\diam(Q) + s\\
&\leq \delta + s \label{eq:Ldiam}
\end{align}

It follows from \eqref{eq:Ki} that $K_i$ is $c\left(\frac{\eta}{\delta s}\right)^{1/m}$-dense in $[0,1]^m$, for some (new) constant $c=c_m>0$. Therefore, for each $i$, the set
$$ \hat{K}_i = \bigcup_{y\in K_i} L_i^y $$
is $c\left(\frac{\eta}{\delta s}\right)^{1/m}$-dense in $Q_0$.

Consider any $p, q\in Q_0$. There is a path $\gamma$ from $p$ to $q$ in $Q_0$ described as follows:
\begin{enumerate}[(i)]
\item\label{enum1} Travel along a segment of length $\leq c \left(\frac{\eta}{\delta s}\right)^{1/m}$ from $p$ to a line $L_1^y \subseteq \hat{K}_1$,
\item\label{enum2} Travel along a segment in $L_1^y$ to a point $p_1$ with $\pi_1(p_1) = \pi_1(q)$.
\item\label{enum3} Travel along a segment of length $\leq c \left(\frac{\eta}{\delta s}\right)^{1/m}$ from $p_1$ to a line $L_2^y \subseteq \hat{K}_2$
\item\label{enum4} Travel along a segment in $L_2^y$ to a point $p_2$ with $\pi_2(p_2) = \pi_2(q)$ and $|\pi_1(p_2) - \pi_1(q)|\lesssim \left(\frac{\eta}{\delta s}\right)^{1/m}$.
\item\label{enum5} Repeat steps \eqref{enum3} and \eqref{enum4} for $i=3,4, \dots, 1+m$, finally reaching a point $p_{1+m}$ with
$$ |\pi_i(p_{1+m}) - \pi_i(q)| \lesssim \left(\frac{\eta}{\delta s}\right)^{1/m} \text{ for each } i=1,\dots 1+m$$
\item\label{enum6} Lastly, travel along a segment of length $\lesssim \left(\frac{\eta}{\delta s}\right)^{1/m}$ from $p_{1+m}$ to $q$.
\end{enumerate}
The implied constants here depend only on $m$.

For each of the $m+1$ segments $\gamma_i$ in $\gamma$ lying inside some $L_i^y$ with $y\in K_i$, we have
$$\diam f(\gamma_i) \leq \diam(f(L_i^y)) \leq \delta + s,$$
using \eqref{eq:Ldiam}.

In total, therefore,
\begin{align*}
d(f(p),f(q)) &\leq \diam f(\gamma)\\
&\lesssim \left(\frac{\eta}{\delta s}\right)^{1/m} + (\delta+s)
\end{align*}
Since $p$ and $q$ were arbitrary points in $Q_0$, we have
$$ \diam f(Q_0) \lesssim \left(\frac{\eta}{\delta s}\right)^{1/m} + (\delta+s).$$
If we now set
$$ \delta = s = \eta^{\frac{1}{m+2}}$$
we obtain
$$ \diam f(Q_0) \lesssim \eta^{\frac{1}{m+2}} + \eta^{\frac{1}{m+2}}.$$
This proves Theorem \ref{thm:onedim} (with $\eta \approx \alpha^{m+2}$).
\end{proof}

\section{A counterexample if the image size is not controlled}\label{sec:example}
We show by an explicit construction that the condition $\HH^{n}(f(Q_0))\leq 1$ is necessary for Theorem \ref{thm:hardersarder} to hold with quantitative bounds, as stated. Our example will show that even assuming $\HH^n(f(Q_0)) < \infty$ is not sufficient. Our construction in this section will be in the case $n=m=1$, so that below $Q_0=[0,1]^2$ is the unit cube in $\RR^2$.

\begin{proposition}\label{prop:example}
There is an absolute constant $\delta>0$ with the following property:

For each $k\in\mathbb{N}$, there is a metric space $X$ of finite $\HH^1$-measure and a $1$-Lipschitz map
$$ f:Q_0 \rightarrow X$$
with
\begin{equation}\label{eq:exH11}
\HH^{1,1}_\infty(f, Q_0) \geq \delta > 0,
\end{equation}
and such that if $E_1, \dots, E_M$ are Hard Sard sets for $f$ with constant $C_{Lip}$, and if
$$ \HH^{1,1}_\infty(f, Q_0 \setminus \cup_i E_i) < \delta/2$$
then the number of Hard Sard sets $M$ must satisfy
$$ M \gtrsim_{C_{Lip}} \delta 2^k.$$
\end{proposition}

We begin by describing the construction of $f$, then prove that it has the properties given in Proposition \ref{prop:example}. By rescaling, it suffices to construct a $4$-Lipschitz map with the above properties, and that is what we will do.

Let $k\in\mathbb{N}$ be fixed. We first define the metric space $X$ that our example will map into. The space $X$ will be a one-dimensional simplicial tree of finite length. To be more precise, first let $S$ denote a ``star'' that consists of $2^{k-1}$ copies of the interval $[0,2^{-k}]$ glued at the points $0$. We call these the ``spikes'' of $S$.

Now consider $2^{k-1}$ copies of $S$: $S_0, \dots S_{2^{k-1}-1}$, each of which has its own central vertex $v_i$. The space $X$ then consists of the union of $S_1, \dots, S_{2^k-1}$ along with a copy of $[0,1]$, such that $v_i$ is attached to the point $i/2^{k-1} \in [0,1]$. We equip $X$ with the intrinsic path metric, which makes $X$ a geodesic tree with 
\begin{equation}\label{eq:Xlength}
\HH^1(X)  \approx 2^{k} <\infty.
\end{equation}
See Figure \ref{fig:example} for a picture of $X$ in the case $k=3$.

\begin{figure}[ht]
\caption{The space $X$ in the case $k=3$. Each ``spike'' has length $1/4 = 2^{-(k-1)}$.}
\centering
\includegraphics[width=0.7\textwidth]{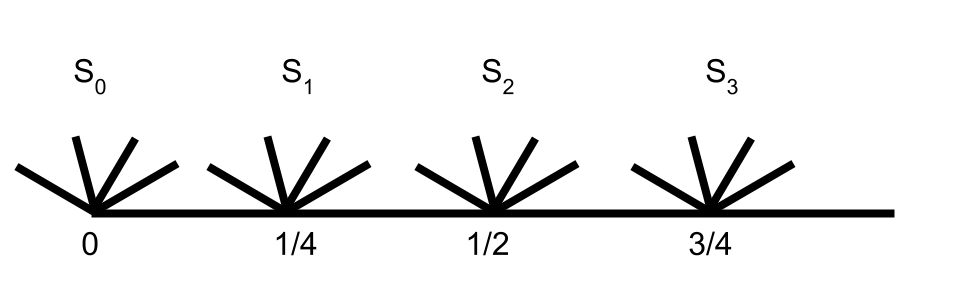}\label{fig:example}
\end{figure}

We now define a mapping $f:Q_0 \rightarrow X$ as follows. First, let
$$ \QQ \subseteq \Delta_{k}$$
denote the collection of dyadic cubes in $\Delta_k$ that are the ``bottom left'' cube in their parent of $\Delta_{k-1}$. 

Thus, if $Q\in\QQ$, then $Q$ is of the form
$$ Q_{a,b} = [(2a)2^{-k}, (2a+1)2^{-k}] \times [(2b)2^{-k}, (2b+1)2^{-k}]$$
where $a,b\in \{0, \dots, 2^{k-1}-1\}$. Let $A$ be the interior of $\cup_{Q\in \QQ} Q$, an open set in $Q_0$.

Let $\pi: \RR^2\rightarrow \RR$ be the projection to the $x$-coordinate. For $Q=Q_{a,b}\in \QQ$, we define $f|_Q = \iota_Q \circ \pi$, where $\iota_Q$ is an isometry that sends the interval $\pi(Q) = [(2a)2^{-k}, (2a+1)2^{-k}]$ to the $b$th spike of star $S_a$, with $\iota_Q((2a)2^{-k}) = v_a$.

It is not hard to see that $f|_A$ is $4$-Lipschitz. Therefore, $f$ extends to a $4$-Lipschitz mapping of $Q_0$ into the tree $X$. (See \cite[section 2.2.2]{AS12}.)

Moreover, $f|_A$ has the following ``coarse injectivity'' property: if $p$ and $p'$ are points in the interiors of distinct cubes of $\QQ$, then $f(p) \neq f(p')$. (They map to different spikes.)

We now argue that the restriction of $f$ to $A$ has mapping content bounded below, independent of $k$.

\begin{lemma}\label{lem:exlowerbound}
There is a constant $\delta>0$, independent of $k$, such that $\HH^{1,1}_\infty(f,A)\geq \delta$.
\end{lemma}
\begin{proof}
Consider any (closed) dyadic cube $R\in\Delta$ that intersects $A$. If $\side(R) \leq 2^{-k}$, then $R$ is contained in a cube $Q\in \QQ$ and so 
$$ \HH^1_\infty(f(R)) \approx \diam(f(R)) = \side(R).$$
If $\side(R) > 2^{-k}$, then there are cubes $Q_{a,b}$ and $Q_{a',b}$ in $\QQ$ that intersect $R$ and have 
$$|a2^{-(k-1)}-a'2^{-(k-1)}| \gtrsim \side(R),$$
with an absolute implied constant. It follows that $f(R)$ intersects the stars $S_a$ and $S_{a'}$, and so
$$ \HH^1_\infty(f(R)) = \diam(f(R)) \geq \dist(S_a, S_{a'}) = |a-a'|2^{-(k-1)} \gtrsim \side(R).$$
Thus, for all cubes $R\in \Delta$ that intersect $A$, we have
$$ \HH^1_\infty(f(R)) \gtrsim \side(R).$$
Hence, if $\{R_j\}$ is an arbitrary collection of almost-disjoint dyadic cubes covering $A$, we have
$$ \sum_{j} \HH^1_\infty(f(R_j)) \side(R_j) \gtrsim \sum_{j: R_j \cap A \neq \emptyset} \side(R_j)^2 \geq |A| = \frac{1}{4}.$$
Therefore
$$ \HH^{1,1}_\infty(f,A) \gtrsim 1.$$
\end{proof}

Next, we argue that no Hard Sard set $E$ can have large intersection with $A$.

\begin{lemma}
If $E\subseteq Q_0$ is a Hard Sard set for $f$, then $|\overline{E} \cap A|\lesssim 2^{-k}$. The implied constant depends only on the Hard Sard constant $C_{Lip}$ for $E$, and not on $k$.
\end{lemma}
\begin{proof}
As $\overline{E}$ is also a Hard Sard set for $f$, we may assume without loss of generality that $\overline{E}=E$.

Let $g$ be the $C_{Lip}$-bi-Lipschitz mapping associated to $E$ and let $F=f\circ g^{-1}$. 

We first observe that if $Q,Q'$ are interiors of cubes in $\QQ$, then
\begin{equation}\label{eq:disjointprojection}
 \pi(g(E\cap Q)) \cap \pi(g(E\cap Q')) = \emptyset.
\end{equation}

Indeed, if $x\in \pi(g(E\cap Q)) \cap \pi(g(E\cap Q'))$, then the vertical line $x\times [0,1]$ intersects both $g(E) \cap g(Q)$ and $g(E) \cap g(Q')$. By Condition \eqref{HS3} of Definition \ref{def:HSpair}, $(x\times [0,1]) \cap g(E)$ is a fiber of $F|_{g(E)}$. It follows that there are points $p\in Q$ and $p'\in Q'$ with $F(g(p)) = F(g(p'))$, i.e., $f(p) = f(p')$. However, by our construction of $f$ this is impossible: the map $f$ sends $Q$ and $Q'$ either to different stars $S_i$ or to two different spikes of the same $S_i$. This proves \eqref{eq:disjointprojection}.

We can therefore compute
\begin{align*}
|E\cap A| &\lesssim |g(E\cap A)|\\
&\leq \sum_{Q \text{ the interior of a cube in } \QQ} |g(E\cap Q)|\\
&\leq \sum_{Q \text{ the interior of a cube in } \QQ} 2^{-k} |\pi(g(E\cap Q))|\\
&= 2^{-k} \sum_{Q \text{ the interior of a cube in } \QQ} |\pi(g(E\cap Q))|\\
&\leq 2^{-k},
\end{align*}
where in the last line we used \eqref{eq:disjointprojection} to bound the sum of the lengths of the \textit{disjoint} sets $\pi(g(E\cap Q)) \subseteq [0,1]$ by $|[0,1]|=1$. 
\end{proof}

We are now ready to complete the proof of Proposition \ref{prop:example}

\begin{proof}[Proof of Proposition \ref{prop:example}]
We refer to the example $f:Q_0\rightarrow X$ defined above. Note that 
$$\HH^1(f(Q_0)) \leq \HH^1(X) < \infty,$$ 
as noted in \eqref{eq:Xlength} and that Lemma \ref{lem:exlowerbound} implies \eqref{eq:exH11}. The mapping $f$ is $4$-Lipschitz, not $1$-Lipschitz, but as noted above this suffices.

Let $E_1, \dots, E_M$ be Hard Sard sets for $f$ with constant $C_{Lip}$ that satisfy
$$ \HH^{1,1}_\infty(f, Q_0 \setminus \cup_i E_i) < \delta/2,$$
where $\delta>0$ is as in Lemma \ref{lem:exlowerbound}.

We then have
\begin{align*}
\delta &\leq \HH^{1,1}_\infty(f,A)\\
&\leq \HH^{1,1}_\infty(f, A\setminus \cup_i E_i) + \sum_{i=1}^M \HH^{1,1}_\infty(f,E_i \cap A)\\
&< \frac{\delta}{2} + CM2^{-k},
\end{align*}
where $C$ is a constant depending only on $C_{Lip}$ and not on $k$.

It follows that $M\gtrsim \delta 2^{k}$, as desired.
\end{proof}

\section{Two versions of mapping content}\label{sec:arbcontent}
In this section, we prove Corollary \ref{cor:arbcontent}. As a reminder, this result concerns the relationship between the notion of mapping content $\HH^{n,m}_\infty$, used throughout the paper, and an alternative version $\hat{\HH}^{n,m}_\infty$ defined in \eqref{eq:arbdef} that uses arbitrary sets rather than dyadic cubes.

We will need the following lemma concerning the types of sets constructed in Proposition \ref{prop:hbilip}. As usual, we write points of $\RR^{n+m}$ as $(x,y)$, where $x\in \RR^n$ and $y\in\RR^m$.
\begin{lemma}\label{lem:arbcontent}
Let $f:Q_0\rightarrow X$ be a $1$-Lipschitz map into a metric space.  Let $E\subseteq Q_0$ and $\phi\colon E \rightarrow Q_0$ a bi-Lipschitz mapping such that the map $\tilde{F}\colon \phi(E) \rightarrow X\times [0,1]^m$ defined by
$$ \tilde{F}(x,y) = (f \circ \phi^{-1}(x,y), y) $$
is bi-Lipschitz on $\phi(E)$.

Then
$$ \hat{\HH}^{n,m}_\infty(f,E) \approx  \HH^{n,m}_\infty(f,E) \approx \HH^{n+m}_\infty(E),$$
with constants depending only on $n$, $m$, and the bi-Lipschitz constants of $\phi$ and $\tilde{F}$.
\end{lemma}
\begin{proof}
By equation \eqref{eq:arbcontent} and Lemma \ref{lem:areaboundscontent}, we have
$$ \hat{\HH}^{n,m}_\infty(f,E) \lesssim_{n,m}  \HH^{n,m}_\infty(f,E) \lesssim_{n,m} \HH^{n+m}_\infty(E).$$
It therefore suffices to show that
$$ \hat{\HH}^{n,m}_\infty(f,E) \gtrsim_{n,m}  \HH^{n+m}_\infty(E).$$

The proof of this is similar to the proof of Lemma \ref{lem:HScontent}. Let $F = f\circ \phi^{-1}$ on $\phi(E)$, so that $\tilde{F}(x,y) = (F(x,y),y)$.

Fix $\epsilon>0$ arbitrary.

Let $\{S_j\}$ be a cover of $E$ by arbitrary sets such that
$$ \sum \HH^n_\infty(f(S_j))\diam(S_j)^m < \hat{\HH}^{n,m}_\infty(f,E) + \epsilon.$$
Note that, without loss of generality, we may assume that each $S_j \subseteq E$, since replacing $S_j$ by $S_j \cap E$ can only decrease the left-hand side of the previous equation.

Let $\{T_i\} \subseteq \{S_j\}$ be an enumeration of those sets $S_j$ in the cover such that $\HH^n_\infty(f(S_j))>0$. We will need the following fact.

\begin{claim}\label{claim:arbcontent}
If $S\in \{S_j\} \setminus \{T_i\}$, then $\HH^{n+m}_\infty(\tilde{F}(\phi(S)))=0$.
\end{claim}
\begin{proof}
We may assume that $\diam(S)>0$ without loss of generality.

The assumption on $S$ implies that $\HH^n_\infty(F(\phi(S))) = \HH^n_\infty(f(S)) =0$. Let $S'=\phi(S)$ and fix $0<\eta<<\diam(S)$. We can cover $F(S')$ by balls $B_k$ such that
$$\sum_k \diam(B_k)^n < \eta.$$

Observe that 
$$ \tilde{F}(S') \subseteq F(S') \times \pi_{\RR^m}(S).$$
For each $k$, we can cover $B_k \times \pi_{\RR^m}(S)$ by $\lesssim \diam(B_k)^{-m}$  balls of diameter $\diam(B_k)$.

Therefore,
$$\HH^{n+m}_\infty(\tilde{F}(\phi(S))) \lesssim \sum_k \diam(B_k)^{-m} \diam(B_k)^{n+m} = \sum_k \diam(B_k)^n < \eta,$$
and sending $\eta$ to zero completes the proof of the claim.
\end{proof}

Let $A_i = \phi(T_i)$ for each $i$, so that $F(A_i)=f(T_i)$. Note that $\diam(A_i) \approx \diam(T_i)$, since $\phi$ is bi-Lipschitz on $E$.

For each $i$, there is a collection of balls $\{B_i^j\}$ covering $F(A_i)$ in $X$ such that
$$ \sum_{j} \diam(B_i^j)^n \leq 2\HH^n_\infty(F(A_i)).$$
Here we are using the fact that $\HH^n_\infty(F(A_i))=\HH^n_\infty(f(T_i))>0$. As in Lemma \ref{lem:HScontent}, we may also assume without loss of generality that $\diam(B_i^j) \lesssim_{n,m} \diam(F(A_i)) \lesssim \diam(A_i)$ for each $j,i$.

Therefore,
\begin{align}
\hat{\HH}^{n,m}_\infty(f, E) &\geq  \sum_j \HH^n_\infty(f(S_j))\diam(S_j)^m - \epsilon\\
&= \sum_i \HH^n_\infty(f(T_i))\diam(T_i)^m - \epsilon\\
&= \sum_i \HH^n_\infty(F(A_i))\diam(T_i)^m - \epsilon\\
&\geq \frac{1}{2} \sum_{i,j} \diam(B_i^j)^n\diam(T_i)^m - \epsilon \label{eq:arblemmabound}
\end{align}

By assumption, the map $\tilde{F}$ defined above is bi-Lipschitz. Thus,
$$ \HH^{n+m}_\infty(\tilde{F}(\phi(E))) \approx \HH^{n+m}_\infty(E).$$

Now, for each fixed $i$, 
$$\tilde{F}(A_i) \subseteq \bigcup_j \left(B_i^j \times \pi_{\RR^m}(A_i)\right). $$

We can cover each $B_i^j \times \pi_{\RR^m}(A_i)$ by 
$$\lesssim \left(\frac{\diam(A_i)}{\diam(B_i^j)}\right)^m \approx \left(\frac{\diam(T_i)}{\diam(B_i^j)}\right)^m$$ 
balls of diameter equal to $\diam(B_i^j)$.

Therefore, 
$$ \HH^{n+m}_\infty(\tilde{F}(A_i)) \lesssim \sum_j \left(\frac{\diam(T_i)}{\diam(B_i^j)}\right)^m\diam(B_i^j)^{n+m} = \sum_j \diam(B_i^j)^n \diam(T_i)^m,$$
and so, using \eqref{eq:arblemmabound},
\begin{align*}
 \hat{\HH}^{n,m}_\infty(f,E) &\geq \frac{1}{2}\sum_{i,j} \diam(B_i^j)^n\diam(T_i)^m -\epsilon\\
&\geq c_1\sum_i \HH^{n+m}_\infty(\tilde{F}(A_i)) - \epsilon\\
&\geq c_2\HH^{n+m}_\infty(\tilde{F}(\phi(E))) - \epsilon \hspace{1in} \text{(using Claim \ref{claim:arbcontent})}\\
&\geq c_3\HH^{n+m}_\infty(E) - \epsilon.
\end{align*}
The constants $c_1, c_2, c_3>0$ depend only on $n$, $m$, and the bi-Lipschitz constants of $\phi$ and $\tilde{F}$. Sending $\epsilon$ to zero completes the proof.
\end{proof}

\begin{proof}[Proof of Corollary \ref{cor:arbcontent}]

Fix $\delta>0$.

Choosing $\alpha=\delta/2$ in Proposition \ref{prop:hbilip} yields constants $N=N(\delta,n,m)$ and $L=L(\delta,n,m)$. The constant $N$ controls the number of sets $F_i$ arising the decomposition of any $1$-Lipschitz map $f:Q_0 \rightarrow X$ in Proposition \ref{prop:hbilip}, and the constant $L$ controls the bi-Lipschitz constants of the mappings $\phi_i$ and 
$$ (x,y) \mapsto (f\circ \phi_i^{-1}(x,y),y)$$
on $F_i$.

Consider any $f:Q_0\rightarrow X$ and any set $A\subseteq Q_0$ with $\HH^{n,m}_\infty(f,A)\geq \delta$. We can then write
$$ Q_0 = F_1 \cup F_2 \cup \dots \cup F_N \cup G,$$
where $F_i$ have the properties given in Proposition \ref{prop:hbilip} and $\HH^{n,m}_\infty(f,G)<\delta/2.$

Let $A_i = A \cap F_i$. On each $A_i$, Lemma \ref{lem:arbcontent} implies that 
$$ \hat{\HH}^{n,m}_\infty(f,A_i) \approx  \HH^{n,m}_\infty(f,A_i) \approx \HH^{n+m}_\infty(A_i),$$
with implied constants depending only on $n$, $m$, and $L=L(n,m,\delta)$.

Thus, at least one $A_{i_0}$ must have 
$$ \hat{\HH}^{n,m}_\infty(f,A_{i_0}) \gtrsim_{\delta,n,m} \HH^{n,m}_\infty(f,A_{i_0}) \geq \frac{1}{N}\left( \HH^{n,m}_\infty(f,A) - \HH^{n,m}_\infty(f,G) \right)  \geq \frac{\delta}{2N}.$$
This proves the corollary.

\end{proof}

\bibliography{hardersarderbib}{}

\begin{thebibliography}{10}

\bibitem{AS12}
J.~Azzam and R.~Schul.
\newblock Hard {S}ard: quantitative implicit function and extension theorems
  for {L}ipschitz maps.
\newblock {\em Geom. Funct. Anal.}, 22(5):1062--1123, 2012.

\bibitem{AS14}
J.~Azzam and R.~Schul.
\newblock A quantitative metric differentiation theorem.
\newblock {\em Proc. Amer. Math. Soc.}, 142(4):1351--1357, 2014.

\bibitem{Da88}
G.~David.
\newblock Morceaux de graphes lipschitziens et int\'{e}grales singuli\`eres sur
  une surface.
\newblock {\em Rev. Mat. Iberoamericana}, 4(1):73--114, 1988.

\bibitem{Davidbook}
G.~David.
\newblock {\em Wavelets and singular integrals on curves and surfaces}, volume
  1465 of {\em Lecture Notes in Mathematics}.
\newblock Springer-Verlag, Berlin, 1991.

\bibitem{DSregular}
G.~David and S.~Semmes.
\newblock Regular mappings between dimensions.
\newblock {\em Publ. Mat.}, 44(2):369--417, 2000.

\bibitem{GCD16}
G.~C. David.
\newblock Bi-{L}ipschitz pieces between manifolds.
\newblock {\em Rev. Mat. Iberoam.}, 32(1):175--218, 2016.

\bibitem{GCDKin}
G.~C. David and K.~Kinneberg.
\newblock Lipschitz and bi-{L}ipschitz maps from {PI} spaces to {C}arnot
  groups.
\newblock {\em \textnormal{to appear}, Indiana Univ. Math. J.}, {P}reprint,
  2017.
\newblock ar{X}iv:1711.03533.

\bibitem{Fe}
H.~Federer.
\newblock {\em Geometric measure theory}.
\newblock Die Grundlehren der mathematischen Wissenschaften, Band 153.
  Springer-Verlag New York Inc., New York, 1969.

\bibitem{GHP}
P.~Goldstein, P.~Haj{\l}asz, and P.~Pankka.
\newblock Topologically nontrivial counterexamples to {S}ard's theorem.
\newblock {\em International Mathematics Research Notices}, 08 2018.
\newblock rny179.

\bibitem{HZ}
P.~Haj{\l}asz and S.~Zimmerman.
\newblock An implicit theorem for lipschitz mappings into metric spaces.
\newblock {\em Indiana Univ. Math. J.}, 69:205--228, 2020.

\bibitem{He}
J.~Heinonen.
\newblock {\em Lectures on analysis on metric spaces}.
\newblock Universitext. Springer-Verlag, New York, 2001.

\bibitem{Jo88}
P.~W. Jones.
\newblock Lipschitz and bi-{L}ipschitz functions.
\newblock {\em Rev. Mat. Iberoamericana}, 4(1):115--121, 1988.

\bibitem{Ka}
R.~Kaufman.
\newblock A singular map of a cube onto a square.
\newblock {\em J. Differential Geom.}, 14(4):593--594 (1981), 1979.

\bibitem{Ki16}
K.~Kinneberg.
\newblock Discrete length-volume inequalities and lower volume bounds in metric
  spaces.
\newblock {\em Math. Z.}, 282(3-4):747--768, 2016.

\bibitem{Kirchheim}
B.~Kirchheim.
\newblock Rectifiable metric spaces: local structure and regularity of the
  {H}ausdorff measure.
\newblock {\em Proc. Amer. Math. Soc.}, 121(1):113--123, 1994.

\bibitem{LLR}
E.~Le~Donne, S.~Li, and T.~Rajala.
\newblock Ahlfors-regular distances on the {H}eisenberg group without
  bi{L}ipschitz pieces.
\newblock {\em Proc. Lond. Math. Soc. (3)}, 115(2):348--380, 2017.

\bibitem{Li}
S.~Li.
\newblock Bi{L}ipschitz decomposition of {L}ipschitz maps between {C}arnot
  groups.
\newblock {\em Anal. Geom. Metr. Spaces}, 3(1):231--243, 2015.

\bibitem{Meyerson}
W.~Meyerson.
\newblock Lipschitz and bilipschitz maps on {C}arnot groups.
\newblock {\em Pacific J. Math.}, 263(1):143--170, 2013.

\bibitem{Sc09}
R.~Schul.
\newblock Bi-{L}ipschitz decomposition of {L}ipschitz functions into a metric
  space.
\newblock {\em Rev. Mat. Iberoam.}, 25(2):521--531, 2009.

\bibitem{Se00}
S.~Semmes.
\newblock Measure-preserving quality within mappings.
\newblock {\em Rev. Mat. Iberoamericana}, 16(2):363--458, 2000.

\bibitem{WY}
S.~Wenger and R.~Young.
\newblock Lipschitz homotopy groups of the {H}eisenberg groups.
\newblock {\em Geom. Funct. Anal.}, 24(1):387--402, 2014.

\bibitem{Zust}
R.~Z\"{u}st.
\newblock Some results on maps that factor through a tree.
\newblock {\em Anal. Geom. Metr. Spaces}, 3(1):73--92, 2015.

\end{thebibliography}
\bibliographystyle{plain}

\end{document}